\documentclass{amsart}
\usepackage[margin=1 in]{geometry}
\usepackage[T1]{fontenc}							
\usepackage[english]{babel}							
\usepackage[utf8]{inputenc}							
\usepackage{amsmath,amsthm, amsfonts,amsrefs,amssymb, enumerate}
\usepackage{mathrsfs}
\usepackage{mathtools}
\mathtoolsset{mathic} 
\usepackage{tikz}

\usepackage{bm} 
\usepackage{bbm} 
\usepackage{mathrsfs} 
\usepackage{dsfont}
\usepackage{enumitem}
\usepackage{bbm}

\newtheorem{theorem}{Theorem}[section]
\newtheorem{lemma}[theorem]{Lemma} 
\newtheorem{corollary}[theorem]{Corollary} 
\newtheorem{proposition}[theorem]{Proposition} 
 
\newtheorem{conjecture}[theorem]{Conjecture} 
\theoremstyle{definition}

\newtheorem{remark}[theorem]{Remark} 

\definecolor{pink}{rgb}{1,0,1}
\def\beq{\begin{eqnarray}}
\def\eeq{\end{eqnarray}}
\newcommand{\nn}{\nonumber}

\newcommand{\vol}{\operatorname{vol}}

\newcommand*{\R}{\mathbb{R}}
\newcommand*{\N}{\mathbb{N}}

\newcommand*{\Z}{\mathbb{Z}}

\newcommand{\pa}{\partial}

\newcommand{\zetasq} {\zeta_{\square}}
\newcommand{\heatsqd} {H_\square ^D}
\newcommand{\heatsqn} {H_\square ^N}
\newcommand{\zetaeq} {\zeta_{\nabla}}
\newcommand{\heateqd} {H_\nabla ^D}
\newcommand{\heateqn} {H_\nabla ^N}
\newcommand{\zetart} {\zeta_{\diamondsuit}}
\newcommand{\heatrtd} {H_{\diamondsuit} ^D}
\newcommand{\heatrtn} {H_{\diamondsuit} ^N}
\newcommand{\zetahemi} {\zeta_{\heartsuit}}
\newcommand{\heathemid} {H_{\heartsuit} ^D}
\newcommand{\heathemin} {H_{\heartsuit} ^N}

\usepackage{graphicx}								
\usepackage{subfig}									
\usepackage{listings}								
\usepackage{eso-pic}								

\usepackage{float} 									

\usepackage[labelfont=bf, textfont=normal,
			justification=justified,
			singlelinecheck=false]{caption}

\usepackage{hyperref}								
\hypersetup{colorlinks, citecolor=black,
   		 	filecolor=black, linkcolor=black,
    		urlcolor=black}

\usepackage[textsize=tiny]{todonotes}   

\usepackage{longtable}

\title{Spectral invariants of integrable polygons}
\author[G.\@ M\aa rdby]{Gustav M\aa rdby}
\address{Mathematical Sciences, Chalmers University of Technology and University
of Gothenburg, 412 96 Gothenburg, Sweden}
\email{mardby@chalmers.se}

\author[J.\@ Rowlett]{Julie Rowlett}
\address{Mathematical Sciences, Chalmers University of Technology and University
of Gothenburg, 412 96 Gothenburg, Sweden}
\email{julie.rowlett@chalmers.se}

\begin{document}
\maketitle
\begin{abstract}
An integrable polygon is one whose interior angles are fractions of $\pi$; that is to say of the form $\frac \pi n$ for positive integers $n$.  We consider the Laplace spectrum on these polygons with the Dirichlet and Neumann boundary conditions, and we obtain new spectral invariants for these polygons.  This includes new expressions for the spectral zeta function and zeta-regularized determinant as well as a new spectral invariant contained in the short-time asymptotic expansion of the heat trace.  Moreover, we demonstrate relationships between the short-time heat trace invariants of general polygonal domains (not necessarily integrable) and smoothly bounded domains and pose conjectures and further related directions of investigation. 
\end{abstract}



\section{Introduction} \label{s:intro}
Let $\Omega \subset \R^2$ be a bounded domain in the Euclidean plane. We consider the Laplace eigenvalue problem, also known as Helmholtz's equation, that is to find all eigenfunctions $u$ and eigenvalues $\lambda$ such that 
\begin{equation} \label{eq:def_laplacian}
    \Delta u + \lambda u = 0 \text{ in } \Omega, \quad \Delta = \partial_x^2 + \partial_y^2.
\end{equation}
The Helmholtz equation is perhaps the most fundamental partial differential equation of mathematics and physics, with numerous real-world applications including acoustic design, structural engineering, diffusion processes, quantum mechanics, and wave propagation.  The function $u$ is assumed to be in the Sobolev space $H^2(\Omega)$, and is further required to satisfy a boundary condition.  Here, we consider the Dirichlet or Neumann boundary conditions, which respectively require the function or its normal derivative to vanish on the boundary.  The set of Laplace eigenvalues is the \em spectrum, \em and any quantity that is defined in terms of the spectrum is a \em spectral invariant.  \em   

Although there exist fast and accurate methods for numerically calculating Laplace eigenvalues \cite{colbrook}, the collection of domains for which the eigenvalues can be computed analytically in closed form is quite limited.  Restricting to polygonal domains in the plane, this collection includes rectangles, equilateral triangles, isosceles right triangles, and hemi-equilateral triangles, also known as 30-60-90 triangles. By \cite{gutkin1986billiards} these are precisely the polygons which are integrable, meaning they have all interior angles of the form $\pi/n$ where $n \in \N$.  It is a straightforwad exercise in planar geometry to prove that all integrable polygons must be one of these four types as shown in Figures \ref{fig:rect} --  \ref{fig:hemiq}.  By \cite{mccartin2008polygonal} (see also \cite[Thm. 1]{rowlett2021crystallographic}), the polygonal domains which strictly tessellate the plane are precisely rectangles, equilateral triangles, isosceles right triangles, and the 30-60-90 triangle. Therefore, polygons being integrable is equivalent to strictly tessellating the plane.  

\begin{figure} [h!]
    \begin{tikzpicture}
    \draw (0,-1) rectangle (8,2);

    \node at (4,-1.3) {$a$};
    
    \node at (-0.3,1/2) {$b$};

    \draw[->] (2,-1) -- (2,0); 
    \draw (2,-1) -- (2,2);     %
    \draw (2,0) -- (2,1);     %
    \draw[->] (2,2) -- (2,1); 

    \filldraw[black] (2,-1) circle (2pt);  
    \filldraw[black] (2,2) circle (2pt);  
\end{tikzpicture} 
\caption{Rectangles are integrable polygons, as their interior angles all measure $\frac \pi 2$.  If the rectangle has sides of lengths $a$ and $b$, then the length of the shortest closed geodesic is twice the length of the shortest side, corresponding to the orbit running perpendicularly between the two longer sides}
\label{fig:rect} 
\end{figure}
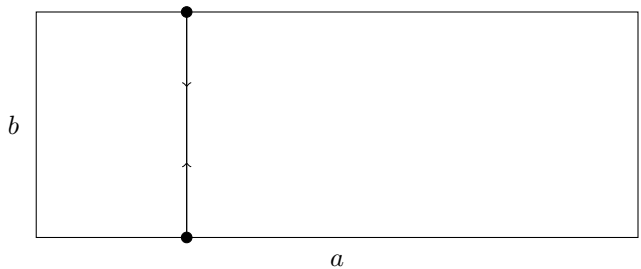

The Laplace eigenfunctions and eigenvalues of rectangles may be obtained using separation of variables by solving the one-dimensional Helmholtz equation.  This was known to mathematicians and physicists in the 18th century; however proving that \em all \em eigenfunctions and eigenvalues are obtained by this method could not be demonstrated until functional analysis was developed in the 19th century \cite{folland2009fourier}. 
The eigenfunctions of a square that are odd along a diagonal produce Dirichlet eigenvalues of isosceles right triangles \cite{kuttler1984eigenvalues}.  For equilateral triangles, Lamé was the first to obtain expressions for eigenvalues and eigenfunctions  \cite{lame1833memoire, lame1852leccons, lame1861leccons}.  At the end of the 19th century these eigenvalues and eigenfunctions were studied further by Pockels, who also noticed that the eigenfunctions of a regular rhombus and a regular hexagon are not trigonometric, i.e. cannot be expressed in terms of sines and cosines \cite{pockels1891uber}.  It was not until 2008 that McCartin proved that in fact, the only polygonal domains that have a complete set of trigonometric eigenfunctions are the integrable polygonal domains \cite{mccartin2008polygonal}.  Rowlett et. al. generalized this result to higher dimensions, where polygons are replaced by polytopes \cite{rowlett2021crystallographic}.  

\begin{figure} [h!]
    \begin{tikzpicture}

    \draw (2,-5.5) -- (6,-5.5) -- (4,{4*sqrt(3)/2-5.5}) -- cycle;
    \draw[->] (4,-5.5) -- (4.5,{sqrt(3)/2-5.5});
    \draw (4.5,{sqrt(3)/2-5.5}) -- (5,{sqrt(3)-5.5});
    \draw[->] (5,{sqrt(3)-5.5}) -- (4,{sqrt(3)-5.5});
    \draw (4,{sqrt(3)-5.5}) -- (3,{sqrt(3)-5.5});
    \draw[->] (3,{sqrt(3)-5.5}) -- (3.5,{sqrt(3)/2-5.5});
    \draw (3.5,{sqrt(3)/2-5.5}) -- (4,-5.5);
    
    \filldraw[black] (4,-5.5) circle (2pt);
    \filldraw[black] (5,{sqrt(3)-5.5}) circle (2pt);
    \filldraw[black] (3,{sqrt(3)-5.5}) circle (2pt);

    \node at (4,-5.8) {$\ell$};
    \node at (5.2,{-5.5+1.2*sqrt(3)}) {$\ell$};
    \node at (2.8,{-5.5+1.2*sqrt(3)}) {$\ell$};
\end{tikzpicture} 
\caption{Equilateral triangles are integrable polygons, as their interior angles all measure $\frac \pi 3$.  Their shortest closed geodesic is formed by connecting the midpoints of the three sides, creating an equilateral triangle with sides of length $\frac \ell 2$ and therewith total length $3 \frac \ell 2$.}
\label{fig:eqtri}

\end{figure}
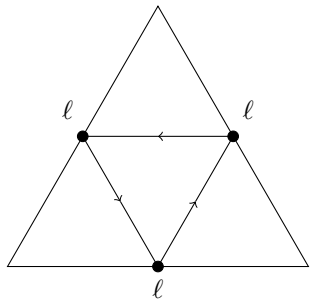 

Although Lam\'e obtained closed formulas for eigenvalues and eigenfunctions of equilateral triangles, similar to the case of rectangles, it is another matter to prove that these are \em all \em eigenvalues.  Indeed, this was first rigorously demonstrated in the 1980s by Pinsky who gave a new way of deriving the eigenvalues and eigenfunctions and established completeness \cite{pinsky1980eigenvalues, pinsky1985completeness}. In 1998, Prager gave a different method for deriving the eigenvalues and eigenfunctions of the equilateral triangle \cite{prager1998eigenvalues}, and in 2002 McCartin gave another elementary method \cite{mccartin2003eigenstructure, mccartin2002eigenstructure, mccartin2004eigenstructure}. McCartin also showed that the eigenfunctions of an equilateral triangle with either two Dirichlet and one Neumann boundary condition or one Dirichlet and two Neumann boundary conditions are not trigonometric \cite{mccartin2003eigenstructure}.  Similar to rectangles and isosceles right triangles, one obtains the eigenvalues of hemi-equilateral (30-60-90) triangles by considering the eigenfunctions of equilateral triangles that are odd along a nodal line \cite{mccartin2003eigenstructure}. 

In 1957 Brownell studied arbitrary polygonal domains with Dirichlet boundary conditions and conjectured that their heat trace expansion only has three terms \cite{brownell1957improved}. A complicated and somewhat incomplete proof of this was given in \cite{bailey1961removal, bailey1962removal}. In 1988 van den Berg and Srisatkunarajah improved this result by obtaining a bound on the error term after the three coefficients \cite{van1988heat}. Using the explicit expressions of the eigenvalues, Verhoeven calculated the heat trace of rectangles, isosceles right triangles, and equilateral triangles in his Bachelor thesis \cite{verhoevencan}.  Although he did not explicitly compute the sharp remainder term, Verhoeven's techniques help us to obtain the sharp remainder term in the short time asymptotic expansion of the heat trace for integrable polygons.  In the case of equilateral triangles, we compute the heat trace via an independent method and show that our expression is in fact equal to Verhoeven's.  In doing so, we are able to prove that a certain expression for the eigenvalues of the equilateral triangle that has appeared in the literature without proof or justification is indeed correct.  Moreover, we obtain expressions for the spectral zeta functions and zeta-regularized determinants of integrable polygons, some of which appear to be new.  By presenting these spectral invariants in their most explicit form, we aim to facilitate both practical computations and theoretical insights. Notably, our investigation of the heat trace leads to a conjecture for a new spectral invariant for convex polygonal domains, namely the length of their shortest closed geodesic.  

\begin{theorem} \label{th:shortest} 
Let $\Omega$ be an integrable polygonal domain.  Let $\{ \lambda_k \}_{k \geq 1}$ denote the Dirichlet eigenvalues of $\Omega$.  Then the heat trace has the short time asymptotic expansion 
\[ \sum_{k \geq 1} e^{-\lambda_k t} = \frac{ |\Omega|}{4\pi t} - \frac{|\pa \Omega|}{8 \sqrt{\pi t}} + \sum_{j=1}^n \frac{\pi^2 - \theta_j^2}{24 \pi \theta_j} + \mathcal O (e^{-\frac{L^2 + \epsilon}{4t}} ), \quad t \to 0, \quad \forall \epsilon > 0.\]
Above, $|\Omega|$ is the area of $\Omega$, $|\pa \Omega|$ is its perimeter, $\theta_j$ are the measures of its interior angles, $n$ is the number of sides, and $L$ is the length of the shortest closed geodesic in $\Omega$.  If instead $\{ \mu_k \}_{k \geq 0}$ denote the Neumann eigenvalues of $\Omega$, then the heat trace 
\[ \sum_{k \geq 0} e^{-\mu_k t} = \frac{ |\Omega|}{4\pi t} + \frac{|\pa \Omega|}{8 \sqrt{\pi t}} + \sum_{j=1}^n \frac{\pi^2 - \theta_j^2}{24 \pi \theta_j} + \mathcal O (e^{-\frac{L^2 + \epsilon}{4t}} ), \quad t \to 0, \quad \forall \epsilon > 0.\]
The remainder estimate in both cases is sharp in the sense that $\frac{L^2}{4}$ is the infimum over all $c>0$ such that the remainder is $\mathcal O (e^{-c/t})$ as $t \to 0$. 
\end{theorem}

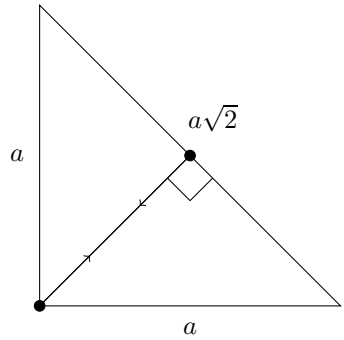
\begin{figure} [h!]
    \begin{tikzpicture}
    \draw (2,-10.5) -- (6,-10.5) -- (2,-6.5) -- cycle;
    \draw[->] (2,-10.5) -- (2.67,-11.33+1.5);
    \draw (2,-10.5) -- (4,-8.5);
    \draw[->] (4,-8.5) -- (3.33,-10.67+1.5);
    \draw (3.33,-10.67+1.5) -- (2.67,-11.33+1.5);
    \draw (3.7,-10.3+1.5) -- (4,-10.6+1.5);
    \draw (4,-10.6+1.5) -- (4.3,-10.3+1.5);
    
    \filldraw[black] (2,-10.5) circle (2pt);
    \filldraw[black] (4,-8.5) circle (2pt);

    \node at (4,-12.3+1.5) {$a$};
    \node at (1.7,-8.5) {$a$};
    \node at (4.3,-8) {$a\sqrt{2}$};

    \end{tikzpicture} 
\caption{Isosceles right triangles are integrable polygons, as their interior angles measure $\frac \pi 2$ and $\frac \pi 4$.  Their shortest closed geodesic is the altitude that joins the right angle to the hypotenuse.}
\label{fig:isoright}
\end{figure}

This result is not surprising if one considers flat tori, as their heat trace consists of a leading term involving the volume, and an exponentially decaying remainder term with exponent of the same form as the remainder here for integrable polygons.  In the following sections, \S \ref{s:rectangles}--\S \ref{s:306090} we calculate the spectral zeta function, zeta-regularized determinant, and heat trace of rectangles, equilateral triangles, isosceles right triangles, and hemi-equilateral triangles.  In the last section of the article \S \ref{s:comparison} we present a brief comparison of the heat traces of flat tori, Euclidean space forms, convex polytopes, convex polygonal domains, and smoothly bounded domains.  We conjecture that for Euclidean space forms as well as convex polytopes, the heat trace has a short time asymptotic expansion with a rapidly decaying remainder term of the same type as that of flat tori and integrable polygons.  We then consider convex polygonal domains converging in the Hausdorff sense to a smoothly bounded domain and prove that the first three heat trace invariants converge to those of the smoothly bounded domain.  In contrast, if smoothly bounded domains converge in the Hausdorff sense to a convex polygonal domain, then only the first two heat trace invariants converge; we show that the third does not. We hope to provide an inclusive introduction to the Laplace eigenvalue problem suitable for a broad readership and at the same time, inspire those readers well-versed in the field to investigate the many remaining open problems. 

\begin{figure} [h!]
    \begin{tikzpicture}

    \draw (2,-16.5) -- (5,-16.5) -- (2,{-16.5+3*sqrt(3)}) -- cycle;
    \draw (2,-16.5) -- (17/4,{2.25/sqrt(3)-16.5});
    \draw[->] (2,-16.5) -- (11/4, {-16.5+0.75/sqrt(3)});
    \draw[->] (17/4,{2.25/sqrt(3)-16.5}) -- (7/2, {-16.5+sqrt(3)/2});
    \draw (11/4, {-16.5+0.75/sqrt(3)}) -- (7/2, {-16.5+sqrt(3)/2});
    \draw (31/8, {-16.5+1.875/sqrt(3)}) -- ({31/8+0.375/sqrt(3)},{-16.5-3/8+1.875/sqrt(3)});
    \draw ({31/8+0.375/sqrt(3)},{-16.5-3/8+1.875/sqrt(3)}) -- ({17/4+sqrt(3)/8}, {-16.5-3/8+0.75*sqrt(3)});
    
    \filldraw[black] (2,-16.5) circle (2pt);
    \filldraw[black] (17/4,{2.25/sqrt(3)-16.5}) circle (2pt);

    \node at (7/2, -18.3+1.5) {$\frac{\ell}{2}$};
    \node at (1.5, {-16.5+1.5*sqrt(3)}) {$\frac{\ell\sqrt{3}}{2}$};
    \node at (4.7, -16.44+1.5) {$\ell$};

\end{tikzpicture}
\caption{Hemi-equilateral triangles are integrable polygons, as their interior angles measure $\frac \pi 2$ and $\frac \pi 3$ and $\frac \pi 6$.  Their shortest closed geodesic is the altitude that joins the right angle to the hypotenuse.}
\label{fig:hemiq}

\end{figure}
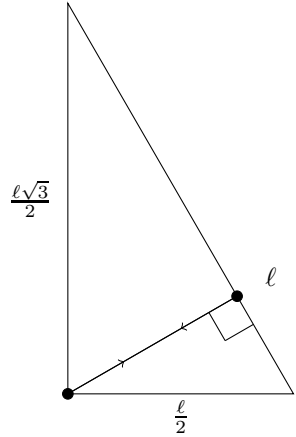

\section*{Acknowledgments} 
JR is grateful to Michiel van den Berg, Sebastian G\"otte, and Klaus Kr\"oncke for insightful suggestions.  Both authors thank Norm Hurt for correspondence that inspired this work.

\newpage
\section*{Notations and abbreviations} 
For the reader's convenience, we include a summary of our notations and abbreviations. \\ 

\begin{tabular}{|l|l|}
\hline $\Delta$ & the Laplace operator or Laplacian defined in \eqref{eq:def_laplacian} \\ 
\hline DBC & Dirichlet boundary condition \\ 
\hline NBC & Neumann boundary condition \\ 
\hline $\zeta_\square$ & the spectral zeta function for a rectangle $[0,a] \times [0,b]$   \\ 
\hline $G_\square(s)$ & the function defined in \eqref{eq:Gsquare} \\  
\hline $H_\square ^D$ & the Dirichlet heat kernel for a rectangle $[0,a] \times [0,b]$ \\  
\hline $H_\square ^N$ & the Neumann heat kernel for a rectangle $[0,a] \times [0,b]$  \\ 
\hline $\zeta_\blacksquare$ & the spectral zeta function for a square with sides of length $a$ \\

\hline 
$\zeta_{\nabla}$ & the spectral zeta function for an equilateral triangle with sides of length $\ell$   \\ 
\hline  
$G_\nabla$ & the function defined in \eqref{eq:Gtri} \\ 
\hline $H_{\nabla} ^D$ & the Dirichlet heat kernel for an equilateral triangle with sides of length $\ell$   \\ 
\hline $H_{\nabla} ^N$ & the Neumann heat kernel for an equilateral triangle with sides of length $\ell$   \\ 

\hline 
$\zeta_{\diamondsuit}$ & the spectral zeta function for an isosceles right triangle with legs of length $a$   \\ 
\hline $G_\diamondsuit$ & the function defined in \eqref{eq:Gdiamond} \\ 
\hline $H_{\diamondsuit} ^D$ & the Dirichlet heat kernel for an isosceles right triangle with legs of length $a$  \\ 
\hline $H_{\diamondsuit} ^N$ & the Neumann heat kernel for an isosceles right triangle with legs of length $a$  \\ 

\hline 
$\zetahemi$ & the spectral zeta function for a hemi-equilateral (30-60-90) triangle with hypotenuse $\ell$   \\ 
\hline $G_\heartsuit$ & the function defined in \eqref{eq:Gheart} \\ 
\hline $\heathemid$ & the Dirichlet heat kernel for a hemi-equilateral (30-60-90) triangle with hypotenuse $\ell$   \\ 
\hline $\heathemin$ & the Neumann heat kernel for a hemi-equilateral (30-60-90) triangle with hypotenuse $\ell$   \\ 
\hline $\zeta_R$ & the Riemann zeta function \\ 
\hline $\eta$ & the Dedekind eta function \\ 

\hline 
\end{tabular}

\section{Spectral invariants of rectangles} \label{s:rectangles}
Consider a rectangular domain $[0,a] \times [0,b]$ in the plane.  One can separate variables and solve the Laplace eigenvalue equation on the two segments $[0,a]$ and $[0,b]$, 
\beq u_{xx} + u_{yy} + \lambda u(x,y) = 0, \quad 0 < x < a, \quad 0<y<b, \label{eq:laplace_rect} \eeq 
imposing either the Dirichlet or Neumann boundary condition, respectively,  
\[ (\text{DBC}): u(0, y) = u(a,y) = u(x,0) = u(x,b) = 0, \]
\[ (\text{NBC}): u_x(0,y) = u_x(a,y) = u_y(x,0) = u_y (x,b) = 0.\]
In the Dirichlet case, the results of this calculation yields the eigenvalues and eigenfunctions 
\begin{equation}
    \lambda_{m,n} = \pi^2\left(\frac{m^2}{a^2} + \frac{n^2}{b^2}\right), \,\, m,n \geq 1, \quad u_{m,n}(x,y) = \sin(m\pi x/a) \sin(n \pi y/b). \label{eq:rect_ev_dbc}
\end{equation}
Having obtained these eigenvalues and eigenfunctions using separation of variables, one may follow \cite[p. 83]{larsson2003partial} to prove completeness, in other words that these are indeed \em all \em eigenfunctions and eigenvalues.  

\subsection{The spectral zeta function and zeta-regularized determinant of rectangles} \label{ss:zeta_rect}
By our calculation of the eigenvalues under the Dirichlet boundary condition in \eqref{eq:rect_ev_dbc} the spectral zeta function is 
\begin{equation*}
    \zetasq(s) = \frac{1}{\pi^{2s}}\sum_{m = 1}^\infty\sum_{n = 1}^\infty \frac{1}{(\frac{m^2}{a^2} + \frac{n^2}{b^2})^s}.
\end{equation*}
We begin by computing two equivalent expressions for the spectral zeta function.  These expressions are relatively simple to obtain using known results, but it seems that the first expression is less widespread.   

\begin{proposition} \label{prop:zeta_det_rect} The spectral zeta function of a rectangle of dimensions $a \times b$ with the Dirichlet boundary condition is equivalently given by the expressions 
\begin{align*}
    \zetasq(s) &= \frac{1}{2}\left(\frac{b}{\pi}\right)^{2s}\bigg[-\zeta_R(2s) + \frac{a\sqrt{\pi}}{b}\frac{\zeta_R(2s-1)\Gamma(s-1/2)}{\Gamma(s)}\bigg] \\ &+ \left(\frac{ab}{\pi}\right)^s\frac{1}{\Gamma(s)}\sqrt{\frac{a}{b}}\sum_{n=1}^\infty n^{s-1/2}\sum_{d|n}d^{1-2s} \int_0^\infty x^{s-3/2}e^{-\pi an(x+x^{-1})/b} dx
\end{align*}
and 
\begin{equation*}
    \zetasq(s) = \frac{a^{2s}}{4}\bigg[G_\square(s) - 2\left(\frac{b}{a\pi}\right)^{2s}\zeta_R(2s) - \frac{2}{\pi^{2s}}\zeta_R(2s)\bigg]. 
\end{equation*}
Here $\zeta_R$ denotes the Riemann zeta function, and 
\begin{equation}
    G_\square(s) = \sum_{m \in \Z} \sum_{n \in \Z} \frac{1}{\pi^{2s}|m + nz|^{2s}}, \quad z = ai/b.  \label{eq:Gsquare}
\end{equation}
\end{proposition}
\begin{proof}
Note that
\begin{align*}
    &\frac{1}{\pi^{2s}}\sum_{m \in \Z} \sum_{n \in \Z} \frac{1}{(\frac{m^2}{a^2} + \frac{n^2}{b^2})^s} \\ 
    &= \frac{4}{\pi^{2s}} \sum_{m= 1}^\infty\sum_{n=1}^\infty \frac{1}{(\frac{m^2}{a^2} + \frac{n^2}{b^2})^s} + 2\left(\frac{a^2}{\pi^2}\right)^s\zeta_R(2s) + 2\left(\frac{b^2}{\pi^2}\right)^s\zeta_R(2s) \\
    &= 4\zetasq(s) + \frac{2\zeta_R(2s)(a^{2s} + b^{2s})}{\pi^{2s}}.
\end{align*}

Here, we use the slight abuse of notation 
\beq  \sum_{m \in \Z} \sum_{n \in \Z} = \sum_{(m,n) \in \Z \times \Z \setminus \{(0,0)\}}. \label{eq:abuse} \eeq
  Consequently, 
\begin{equation*}
    \zetasq(s) = \frac{1}{\pi^{2s}}\sum_{m=1}^\infty \sum_{n=1}^\infty \frac{1}{(\frac{m^2}{a^2} + \frac{n^2}{b^2})^s} = \frac{1}{4\pi^{2s}}\bigg[\sum_{m \in \Z} \sum_{n \in \Z} \frac{1}{(\frac{m^2}{a^2} + \frac{n^2}{b^2})^s} - 2\zeta_R(2s)(a^{2s}+b^{2s})\bigg].
\end{equation*}
By \cite[p. 87]{selberg1967epstein}, we have
\begin{align*}
    &\sum_{m \in \Z} \sum_{n \in \Z} \frac{1}{(\frac{m^2}{a^2} + \frac{n^2}{b^2})^s} = 2a^{2s}\zeta_R(2s) + \frac{2ab^{2s-1}\sqrt{\pi}\zeta_R(2s-1)\Gamma(s-1/2)}{\Gamma(s)} + Q(s), \\
    &Q(s) = \frac{4(\pi ab)^s}{\Gamma(s)}\sqrt{\frac{a}{b}} \sum_{n=1}^\infty n^{s-1/2}\sum_{d|n}d^{1-2s} \int_0^\infty x^{s-3/2}e^{-\pi an(x+x^{-1})/b} dx.
\end{align*}
Therefore,
\begin{align*}
    \zetasq(s) &= \frac{1}{2}\left(\frac{b}{\pi}\right)^{2s}\bigg[-\zeta_R(2s) + \frac{a\sqrt{\pi}}{b}\frac{\zeta_R(2s-1)\Gamma(s-1/2)}{\Gamma(s)}\bigg] \\ &+ \left(\frac{ab}{\pi}\right)^s\frac{1}{\Gamma(s)}\sqrt{\frac{a}{b}}\sum_{n=1}^\infty n^{s-1/2}\sum_{d|n}d^{1-2s} \int_0^\infty x^{s-3/2}e^{-\pi an(x+x^{-1})/b} dx.
\end{align*}
To obtain the second expression for $\zetasq(s)$ we use the Dedekind eta function 
\begin{equation}
    \eta(\tau) = q^{1/12}\prod_{n=1}^\infty (1 - q^{2n}), \,\, q = e^{\pi i \tau}, \,\, \text{Im}(\tau) > 0. \label{eq:dedekind}
\end{equation}
Then, 
\begin{align*}
    \zetasq(s) = \sum_{m=1}^\infty \sum_{n=1}^\infty \frac{1}{\pi^{2s}(\frac{m^2}{a^2} + \frac{n^2}{b^2})^s} = a^{2s} \sum_{m=1}^\infty \sum_{n=1}^\infty \frac{1}{\pi^{2s}(m^2 + \frac{n^2a^2}{b^2})^s} = a^{2s} \sum_{m=1}^\infty \sum_{n=1}^\infty \frac{1}{\pi^{2s}|m + nz|^{2s}},
\end{align*}
where $z = ai/b$. Since
\begin{align*}
    \sum_{m \in \Z} \sum_{n \in \Z} \frac{1}{\pi^{2s}|m + nz|^{2s}} = 4\sum_{m=1}^\infty \sum_{n=1}^\infty \frac{1}{\pi^{2s}|m + nz|^{2s}} + 2\left(\frac{b}{a\pi}\right)^{2s}\zeta_R(2s) + \frac{2}{\pi^{2s}}\zeta_R(2s),
\end{align*}
it follows that
\begin{align*}
    \sum_{m=1}^\infty \sum_{n=1}^\infty \frac{1}{\pi^{2s}|m + nz|^{2s}} = \frac{1}{4}\bigg[\sum_{m \in \Z} \sum_{n \in \Z} \frac{1}{\pi^{2s}|m + nz|^{2s}} - 2\left(\frac{b}{a\pi}\right)^{2s}\zeta_R(2s) - \frac{2}{\pi^{2s}}\zeta_R(2s)\bigg].
\end{align*}
If we write 
\begin{equation*}
    G_\square(s) = \sum_{m \in \Z} \sum_{n \in \Z} \frac{1}{\pi^{2s}|m + nz|^{2s}},
\end{equation*}
then we obtain
\begin{equation*}
    \zetasq(s) = \frac{a^{2s}}{4}\bigg[G_\square(s) - 2\left(\frac{b}{a\pi}\right)^{2s}\zeta_R(2s) - \frac{2}{\pi^{2s}}\zeta_R(2s)\bigg].
\end{equation*}
\end{proof} 
These expressions allow one to meromorphically extend the spectral zeta function to the complex plan and evaluate it at points of interest.  In particular, as a corollary we obtain two expressions for the zeta-regularized determinant. 

\begin{corollary} \label{cor:zeta_det_rect}
The zeta-regularized determinant of a rectangle of dimensions $a \times b$ with the Dirichlet boundary condition is $e^{-\zetasq'(0)}$ with $\zetasq'(0)$ equivalently given by 
\begin{align*}
    \zetasq'(0) &= \frac{1}{2}\log(2b) + \frac{\pi a}{12 b} + \sum_{n=1}^\infty \frac{1}{ne^{2\pi an/b}}\sum_{d|n} d, \\
    \zetasq'(0) &= \frac{1}{2}\log\left(\frac{2b}{|\eta(z)|^2}\right)= \frac{1}{2}\log(2b) + \frac{\pi a}{12 b} - \sum_{n=1}^\infty \log(1- e^{-2\pi n a/b}), \quad z = ai/b. 
\end{align*}
\end{corollary} 
\begin{proof} 
We differentiate our first expression from Proposition \ref{prop:zeta_det_rect} and obtain
\begin{align*}
    \zetasq'(s) &= \log\left(\frac{b}{\pi}\right)\left(\frac{b}{\pi}\right)^{2s}\bigg[-\zeta_R(2s) + \frac{a\sqrt{\pi}}{b}\frac{\zeta_R(2s-1)\Gamma(s-1/2)}{\Gamma(s)}\bigg] \\
    &+ \frac{1}{2}\left(\frac{b}{\pi}\right)^{2s}\bigg[-2\zeta_R'(2s) + \frac{a\sqrt{\pi}}{b}\frac{1}{\Gamma(s)^2}\bigg[
    2\Gamma(s)\zeta_R'(2s-1)\Gamma(s-1/2) \\ &+ \Gamma(s)\zeta_R(2s-1)\Gamma'(s-1/2) 
    - \Gamma'(s)\zeta_R(2s-1)\Gamma(s-1/2)\bigg]\bigg] \\
    &+ \frac{d}{ds}\left[\left(\frac{ab}{\pi}\right)^s\frac{1}{\Gamma(s)}\sqrt{\frac{a}{b}}\sum_{n=1}^\infty n^{s-1/2}\sum_{d|n}d^{1-2s} \int_0^\infty x^{s-3/2}e^{-\pi an(x+x^{-1})/b} dx\right].
\end{align*}
We write 
\begin{align*}
    &\frac{d}{ds}\left[\left(\frac{ab}{\pi}\right)^s\frac{1}{\Gamma(s)}\sqrt{\frac{a}{b}}\sum_{n=1}^\infty n^{s-1/2}\sum_{d|n}d^{1-2s} \int_0^\infty x^{s-3/2}e^{-\pi an(x+x^{-1})/b} dx\right] \\
    = &-\left(\frac{ab}{\pi}\right)^s\frac{\Gamma'(s)}{\Gamma(s)^2}\sqrt{\frac{a}{b}}\sum_{n=1}^\infty n^{s-1/2}\sum_{d|n}d^{1-2s} \int_0^\infty x^{s-3/2}e^{-\pi an(x+x^{-1})/b} dx \\
    &+ \frac{1}{\Gamma(s)}\frac{d}{ds}\left[\left(\frac{ab}{\pi}\right)^s\sqrt{\frac{a}{b}}\sum_{n=1}^\infty n^{s-1/2}\sum_{d|n}d^{1-2s} \int_0^\infty x^{s-3/2}e^{-\pi an(x+x^{-1})/b} dx\right].
\end{align*}
In Lemma \ref{lemma_vanishes} we prove that we may differentiate termwise and under the integral in the second term above, and thereby obtain that this term vanishes due to the presence of $\frac 1 \Gamma$. We then obtain 
\begin{align*}
    \zetasq'(0) &= \log\left(\frac{b}{\pi}\right)\bigg[-\zeta_R(0)\bigg] + \frac{1}{2}\bigg[-2\zeta_R'(0) - \frac{a\sqrt{\pi}\Gamma'(0)\zeta_R(-1)\Gamma(-1/2)}{b\Gamma(0)^2}\bigg] \\
    &- \frac{\Gamma'(0)}{\Gamma(0)^2}\sqrt{\frac{a}{b}}\sum_{n=1}^\infty n^{-1/2}\sum_{d|n}d \int_0^\infty x^{-3/2}e^{-\pi an(x+x^{-1})/b} dx.
\end{align*}
Above, the terms involving $\frac{\Gamma'(0)}{\Gamma(0)^2}$ should be interpreted as $-\frac{d}{ds}\frac{1}{\Gamma(s)}$ evaluated at $s=0$. Since $\Gamma$ has a simple pole at zero with
\begin{equation*} 
    \lim_{s\to 0} s\Gamma(s) = 1,
\end{equation*}
it follows that 
\begin{equation} \label{eq:gamma_function_simple_pole}
    \frac{\Gamma'(0)}{\Gamma(0)^2} = -1.
\end{equation}
To compute
\begin{equation*}
    \int_0^\infty x^{-3/2}e^{-\pi an(x+x^{-1})/b} dx,
\end{equation*}
we make the change of variable $x = e^t$ to get
\begin{equation*}
    \int_0^\infty x^{-3/2}e^{-\pi an(x+x^{-1})/b} dx = \int_{-\infty}^\infty e^{-t/2} e^{-2\pi an/b \cosh(t)} dt.
\end{equation*}
Since $e^{-t/2} = \cosh(t/2) - \sinh(t/2)$, we can rewrite this as
\begin{equation*}
    \int_{-\infty}^\infty (\cosh(t/2) - \sinh(t/2))e^{-2\pi an/b \cosh(t)} dt = 2\int_0^\infty \cosh(t/2)e^{-2\pi an/b \cosh(t)} dt, 
\end{equation*}
which by \cite[Eq. 10.32.9]{NIST:DLMF} equals $2K_{1/2}(2\pi a n/b)$. Since $K_{1/2}(z) = \sqrt{\frac{\pi}{2z}} e^{-z}$ (see \cite[Eq. 10.39.2]{NIST:DLMF}), we obtain
\begin{equation} \label{eq:miracle_integral}
    \int_0^\infty x^{-3/2}e^{-\pi an(x+x^{-1})/b} dx = \sqrt{\frac{b}{an}} e^{-2\pi an/b}.
\end{equation}
Moreover, we have
\begin{equation} \label{eq:values_Riemann_zeta_Gamma}
    \zeta_R(0) = -\frac{1}{2}, \,\, \zeta_R'(0) = -\frac{\log(2\pi)}{2}, \,\, \zeta_R(-1) = -\frac{1}{12}, \,\, \Gamma(-1/2) = -2\sqrt{\pi}.
\end{equation}
Thus,
\begin{equation*}
    \zetasq'(0) = \frac{1}{2}\log(2b) + \frac{\pi a}{12 b} + \sum_{n=1}^\infty \frac{1}{ne^{2\pi an/b}}\sum_{d|n} d.
\end{equation*}
Next we differentiate the second expression of Proposition \ref{prop:zeta_det_rect} 
\begin{align*}
    \zetasq'(s) &= a^{2s}\log(a)\left[\frac{G_\square(s)}{2} - \left(\frac{b}{a\pi}\right)^{2s}\zeta_R(2s) - \frac{1}{\pi^{2s}}\zeta_R(2s)\right] \\
    &+ a^{2s}\bigg[\frac{G_\square'(s)}{4} - \left(\frac{b}{a\pi}\right)^{2s}\log\left(\frac{b}{a\pi}\right)\zeta_R(2s) - \left(\frac{b}{a\pi}\right)^{2s}\zeta_R'(2s) + \frac{1}{\pi^{2s}}\log(\pi)\zeta_R(2s) - \frac{1}{\pi^{2s}}\zeta_R'(2s)\bigg],
\end{align*}
which, after inserting $s = 0$, becomes
\begin{align*}
    \zetasq'(0) = \log(a)\left[\frac{G_\square(0)}{2} + 1\right] + \frac{G_\square'(0)}{4} + \frac{1}{2}\log\left(\frac{4b}{a}\right).
\end{align*}
We have by \cite[p. 204-205]{osgood1988extremals} (see also \cite[p. 1830-1831]{aldana2018polyakov}),
\begin{equation*}
    G_\square(0) = -1, \,\, G_\square'(0) = -\frac{1}{12}\log\left((2\pi)^{24}\frac{(\eta(z)\bar{\eta}(z))^{24}}{\pi^{24}}\right),
\end{equation*}
and since
\begin{equation*}
    \log(\eta(z)) = -\frac{\pi a}{12b} + \sum_{n=1}^\infty \log(1 - e^{-2\pi n a/b}),
\end{equation*}
we obtain
\begin{equation*}
    \zetasq'(0) = \frac{1}{2}\log\left(\frac{2b}{|\eta(z)|^2}\right)= \frac{1}{2}\log(2b) + \frac{\pi a}{12 b} - \sum_{n=1}^\infty \log(1- e^{-2\pi n a/b}).
\end{equation*}
It may be of some interest to verify explicitly that our two expressions of $\zetasq'(0)$ are equal.  
To this end, it is enough to show that
\begin{equation*}
    -\sum_{n = 1}^\infty \log(1 - q^n) = \sum_{n = 1}^\infty \frac{q^n}{n} \sum_{d | n} d
\end{equation*}
for $|q| < 1$. We Taylor expand $\log(1-q^n)$ around 0 and obtain
\begin{align*}
    -\sum_{n = 1}^\infty \log(1 - q^n) = -\sum_{n = 1}^\infty \sum_{m = 1}^\infty \frac{(-1)^{m+1}(-q^n)^m}{m} = \sum_{n = 1}^\infty \sum_{m = 1}^\infty \frac{q^{nm}}{m} = \sum_{m = 1}^\infty \frac{1}{m} \sum_{n = 1}^\infty q^{nm}.
\end{align*}
If we write this 
as a power series in $q$, we see that $q^n$ gets a contribution of $1/m$ if and only $m$ divides $n$. Thus, 
\begin{align*}
    -\sum_{n = 1}^\infty \log(1 - q^n) = \sum_{n = 1}^\infty q^n \sum_{d|n} \frac{1}{d} = \sum_{n = 1}^\infty \frac{q^n}{n} \sum_{d|n} \frac{n}{d} = \sum_{n = 1}^\infty \frac{q^n}{n} \sum_{d|n} d,
\end{align*}
where the last equality follows from the fact that $n/d \mapsto d$ is a bijection between the divisors of $n$. This completes the proof that the two expressions for $\zetasq'(0)$ are equal. 
\end{proof}

\subsection{The heat trace of a rectangle and its short time asymptotic expansion} \label{ss:heat_rect}
For the rectangle $[0,a] \times [0,b]$, the Dirichlet heat trace  
\begin{equation*}
    \heatsqd(t) = \sum_{\lambda_{m,n}} e^{-\lambda_{m,n}t} = \sum_{m=1}^\infty \sum_{n=1}^\infty e^{-\pi^2(\frac{m^2}{a^2} + \frac{n^2}{b^2})t} = \sum_{m=1}^\infty e^{-\frac{\pi^2m^2t}{a^2}} \sum_{n=1}^\infty e^{-\frac{\pi^2n^2t}{b^2}}.
\end{equation*}
Since 
\begin{equation*}
    \sum_{n=1}^\infty q^{n^2} = \frac{\Theta_3(q)-1}{2}, \quad \Theta_3(q) = \sum_{n \in \Z} q^{n^2}, 
\end{equation*}
the heat trace 
\begin{equation*}
    \heatsqd(t) = \frac{\bigl(\Theta_3(e^{-\pi^2 t/a^2}) - 1\bigr)\bigl(\Theta_3(e^{-\pi^2 t/b^2}) - 1\bigr)}{4}.
\end{equation*}
Although we will not use this expression, we present it for the sake of completeness.  Our focus here is on the asymptotic expansion of $\heatsqd(t)$ as $t \to 0$. 

For any convex polygon $\Omega$ in the Euclidean plane with interior angles $\gamma_1,\dots,\gamma_n$, its heat trace with Dirichlet boundary conditions admits the asymptotic expansion 
\begin{equation*}
     \frac{|\Omega|}{4\pi t} - \frac{|\partial\Omega|}{8\sqrt{\pi t}} + \sum_{i=1}^n \frac{\pi^2-\gamma_i^2}{24\pi\gamma_i} + \mathcal{O}(e^{-c/t}), \,\, t \to 0,
\end{equation*}
for some $c > 0$ that has been estimated in \cite{van1988heat}.  For the Neumann boundary condition, an analogous estimate with such a remainder term remains an open problem.   Verhoeven \cite{verhoevencan} used Poisson's summation formula to obtain an expression that can be used not only to obtain further terms in the asymptotic expression, but also to determine the infinum of all such $c$ for the remainder estimate above.  We provide the result here as well as the corresponding result for the Neumann boundary condition.  It is interesting to note that certain terms appear with different signs according to the different boundary conditions.  

\begin{theorem} [Verhoeven \cite{verhoevencan}] \label{th:new_heat_rect}
Let $\Omega = [0,a] \times [0,b]$ be a rectangle that is not a square.  Then the heat trace with the Dirichlet boundary condition admits the asymptotic expansion 
\begin{align*}
    \heatsqd(t) &= \frac{ab}{4\pi t} - \frac{a+b}{4\sqrt{\pi t}} + \frac{1}{4} + \frac{ab}{2 \pi t} e^{-\frac{\min(a,b)^2}{t}} - \frac{\min(a,b)}{2\sqrt{\pi t}} e^{-\frac{\min(a,b)^2}{t}}  \\ 
    &+ \mathcal O \left(t^{-1} e^{-\frac{c}{t}}  \right), \quad t \to 0, \quad c = \min ( \max(a, b)^2, 4 \min(a, b)^2 ).
\end{align*}
The heat trace with the Neumann boundary condition admits the asymptotic expansion 
\begin{align*}
    \heatsqn(t) &= \frac{ab}{4\pi t} + \frac{a+b}{4\sqrt{\pi t}} + \frac{1}{4} + \frac{ab}{2 \pi t} e^{-\frac{\min(a,b)^2}{t}} + \frac{\min(a,b)}{2\sqrt{\pi t}} e^{-\frac{\min(a,b)^2}{t}}  \\ 
    &+ \mathcal O \left(t^{-1} e^{-\frac{c}{t}} \right ), \quad t \to 0.
\end{align*}
If the rectangle is a square, $a=b$, then the respective expansions are 
\begin{align*}
    \heatsqd(t) &= \frac{a^2}{4\pi t} - \frac{a}{2\sqrt{\pi t}} + \frac{1}{4} + \frac{a^2}{ \pi t} e^{-\frac{a^2}{t}} - \frac{a}{\sqrt{\pi t}} e^{-\frac{a^2}{t}}  + \mathcal O (t^{-1} e^{-\frac{2a^2}{t}}  ), \quad t \to 0, \\
    \heatsqn(t) &= \frac{a^2}{4\pi t} + \frac{a}{2\sqrt{\pi t}} + \frac{1}{4} + \frac{a^2}{ \pi t} e^{-\frac{a^2}{t}} + \frac{a}{\sqrt{\pi t}} e^{-\frac{a^2}{t}}  + \mathcal O (t^{-1} e^{-\frac{2a^2}{t}}  ), \quad t \to 0.
\end{align*}
In all cases the remainders are sharp. 
\end{theorem} 

\begin{proof}
By Poisson's summation formula (see \cite[Lemma 2.2.2]{verhoevencan})
\begin{align*}
    \sum_{m \in \Z} e^{-\frac{\pi^2m^2}{a^2}t} = \frac{a}{\sqrt{\pi t}}\sum_{m \in \Z} e^{-\frac{m^2a^2}{t}},
\end{align*}
which implies that
\begin{equation*}
    \sum_{m = 1}^\infty e^{-\frac{\pi^2m^2}{a^2}t} = \frac{1}{2}\left(\frac{a}{\sqrt{\pi t}} - 1\right) + \frac{a}{\sqrt{\pi t}} \sum_{m = 1} e^{-\frac{m^2 a^2}{t}}.
\end{equation*}
Consequently, as calculated in \cite{verhoevencan} it is straightforward to show that the Dirichlet heat trace 
\begin{align*}
    \heatsqd(t) &= 
    \left(\frac{1}{2}\left(\frac{a}{\sqrt{\pi t}} - 1\right) + \frac{a}{\sqrt{\pi t}}\sum_{m=1}^\infty e^{-\frac{m^2a^2}{t}}\right)\left(\frac{1}{2}\left(\frac{b}{\sqrt{\pi t}} - 1\right) + \frac{b}{\sqrt{\pi t}}\sum_{n=1}^\infty e^{-\frac{n^2b^2}{t}}\right) = \\
    &= \frac{ab}{4\pi t} - \frac{a+b}{4\sqrt{\pi t}} + \frac{1}{4} + \frac{ab}{2\pi t}\sum_{m = 1}^\infty e^{-\frac{m^2a^2}{t}} + \frac{ab}{2\pi t}\sum_{n = 1}^\infty e^{-\frac{n^2b^2}{t}} \\
    &- \frac{a}{2\sqrt{\pi t}}\sum_{m = 1}^\infty e^{-\frac{m^2a^2}{t}} - \frac{b}{2\sqrt{\pi t}}\sum_{n = 1}^\infty e^{-\frac{n^2b^2}{t}} + \frac{ab}{\pi t} \sum_{m=1}^\infty \sum_{n=1}^\infty e^{-\frac{m^2a^2 + n^2b^2}{t}}.
\end{align*}
The proof in the Dirichlet case is then completed by calculating the leading order terms and determining the remainder by analyzing each of the three series.  The eigenvalues of the rectangle $[0,a] \times [0,b]$ with Neumann boundary conditions are given by 
\begin{equation*}
    \lambda_{m,n} = \pi^2\left(\frac{m^2}{a^2} + \frac{n^2}{b^2}\right), \,\, m,n \geq 0,
\end{equation*}
so the heat trace becomes
\begin{align*}
    \heatsqn(t) &= \sum_{m=0}^\infty\sum_{n=0}^\infty e^{-\pi^2(\frac{m^2}{a^2} + \frac{n^2}{b^2})t} = \sum_{n=0}^\infty e^{-\frac{\pi^2n^2}{b^2}t} + \sum_{m=1}^\infty\left(e^{-\frac{\pi^2m^2}{a^2}t} + \sum_{n=1}^\infty e^{-\pi^2(\frac{m^2}{a^2} + \frac{n^2}{b^2})t}\right) \\
    &= 1 + \sum_{n=1}^\infty e^{-\frac{\pi^2n^2}{b^2}t} + \sum_{m=1}^\infty e^{-\frac{\pi^2m^2}{a^2}t} + \heatsqd(t)   \\ 
    &= 1+ \frac{1}{2}\left(\frac{b}{\sqrt{\pi t}} - 1\right) + \frac{b}{\sqrt{\pi t}}\sum_{n=1}^\infty e^{-\frac{n^2b^2}{t}} + \frac{1}{2}\left(\frac{a}{\sqrt{\pi t}} - 1\right) + \frac{a}{\sqrt{\pi t}}\sum_{m=1}^\infty e^{-\frac{m^2a^2}{t}} + \heatsqd(t) \\ 
    &= \frac{ab}{4\pi t} + \frac{a+b}{4\sqrt{\pi t}} + \frac{1}{4} + \frac{ab}{2\pi t}\sum_{m = 1}^\infty e^{-\frac{m^2a^2}{t}} + \frac{ab}{2\pi t}\sum_{n = 1}^\infty e^{-\frac{n^2b^2}{t}} \\
    &+ \frac{a}{2\sqrt{\pi t}}\sum_{m = 1}^\infty e^{-\frac{m^2a^2}{t}} + \frac{b}{2\sqrt{\pi t}}\sum_{n = 1}^\infty e^{-\frac{n^2b^2}{t}} + \frac{ab}{\pi t} \sum_{m=1}^\infty \sum_{n=1}^\infty e^{-\frac{m^2a^2 + n^2b^2}{t}}.
\end{align*}
Here we have used Poisson's summation formula and our calculation of $\heatsqd(t)$.  The proof in the Neumann case is similarly completed by reading off the leading order terms and collecting the remainder.  

When $a = b$, the heat traces simplify to
\begin{align*}
    \heatsqd(t) &= \frac{a^2}{4\pi t} - \frac{a}{2\sqrt{\pi t}} + \frac{1}{4} + \frac{a^2}{\pi t}\sum_{m = 1}^\infty e^{-\frac{m^2a^2}{t}} 
    - \frac{a}{\sqrt{\pi t}}\sum_{m = 1}^\infty e^{-\frac{m^2a^2}{t}} + \frac{a^2}{\pi t} \sum_{m=1}^\infty \sum_{n=1}^\infty e^{-\frac{(m^2+n^2)a^2}{t}}, \\
    \heatsqn(t) &= \frac{a^2}{4\pi t} + \frac{a}{2\sqrt{\pi t}} + \frac{1}{4} + \frac{a^2}{\pi t}\sum_{m = 1}^\infty e^{-\frac{m^2a^2}{t}} 
    + \frac{a}{\sqrt{\pi t}}\sum_{m = 1}^\infty e^{-\frac{m^2a^2}{t}} + \frac{a^2}{\pi t} \sum_{m=1}^\infty \sum_{n=1}^\infty e^{-\frac{(m^2+n^2)a^2}{t}},
\end{align*}
from which the claimed heat trace expansions follow.
\end{proof}

To compare with the estimates in \cite{van1988heat} and \cite{verhoevencan}, we note that as $t \to 0$,
\begin{align*}
    &\frac{ab}{2\pi t}\sum_{m=1}^\infty e^{-\frac{m^2 a^2}{t}} \text{ and } \frac{a}{2\sqrt{\pi t}}\sum_{m=1}^\infty e^{-\frac{m^2 a^2}{t}} \, \text{ are } \, \mathcal{O}(e^{-\frac{a^2 - \epsilon}{t}}) \, \text{ for any } \,\epsilon > 0 \, \text{ but not } \, \mathcal{O}(e^{-\frac{a^2}{t}}), \\
    &\frac{ab}{2\pi t}\sum_{n=1}^\infty e^{-\frac{n^2 b^2}{t}} \text{ and } \frac{b}{2\sqrt{\pi t}}\sum_{n=1}^\infty e^{-\frac{n^2 b^2}{t}} \, \text{ are } \, \mathcal{O}(e^{-\frac{b^2 - \epsilon}{t}}) \, \text{ for any } \,\epsilon > 0 \, \text{ but not } \, \mathcal{O}(e^{-\frac{b^2}{t}}), \\
    &\frac{ab}{\pi t} \sum_{m=1}^\infty \sum_{n=1}^\infty e^{-\frac{m^2 a^2 + n^2 b^2}{t}} \, \text{ is } \, \mathcal{O}(e^{-\frac{a^2+b^2 - \epsilon}{t}}) \, \text{ for any } \,\epsilon > 0 \, \text{ but not } \, \mathcal{O}(e^{-\frac{a^2+b^2}{t}}).
\end{align*}
The estimate of $c$ one obtains in \cite[Theorem 1]{van1988heat} in the case of a rectangle is $c=\frac{\min(a,b)^2}{64}$.  In particular, there is no minimal $c$ such that the error term is $\mathcal{O}(e^{-\frac{c}{t}})$, but the infimum of all such values is $\min(a,b)^2$.   By \cite[Prop. 8]{hezari2021dirichlet}, $\min(a,b)^2$ is the square of half the length of the shortest closed geodesic in the rectangle.  This is not surprising considering the analogous result one can obtain for flat tori as discussed in \S \ref{s:comparison}.  

We can easily generalize our heat trace result to $n$-dimensional Euclidean boxes. For a Euclidean box, $\prod_{j=1}^n [0,a_j]$, a similar calculation as above yields the heat trace in the Dirichlet and Neumann case
\begin{equation*}
    \frac{1}{2^n}\prod_{j=1}^n\left(\frac{a_j}{\sqrt{\pi t}} - 1\right) + \mathcal{O}(e^{-\frac{\min(a_1,\dots,a_n)^2-\epsilon}{t}}), \quad \forall \epsilon > 0, 
\end{equation*}
\begin{equation*}
     \frac{1}{2^n}\prod_{j=1}^n\left(\frac{a_j}{\sqrt{\pi t}} + 1\right) + \mathcal{O}(e^{-\frac{\min(a_1,\dots,a_n)^2-\epsilon}{t}}), \quad \forall \epsilon > 0,
\end{equation*}
respectively. We again observe that $\min(a_1,\dots,a_n)^2$ is the square of half the length of the shortest closed geodesic in this Euclidean box.

\section{Spectral invariants of equilateral triangles} \label{sec:pinsky_equilateral_triangle}
Let $\Omega = \{(x,y) \in \R^2 : 0 < y < x\sqrt{3}, \,\, y < \sqrt{3}(1-x)\}$ be an equilateral triangular domain.  We note that the sides each have length one.  We consider the Laplace eigenvalue problem with the Dirichlet boundary condition 
\begin{equation} \label{eq:eigenvalue_triangle}
        \begin{cases}
            \Delta f(x,y) + \lambda f(x,y) = 0 \text{ in } \Omega, \\
            f(x,y) = 0 \text{ on } \partial \Omega.
        \end{cases}
\end{equation}
By \cite[Thm. 1]{pinsky1980eigenvalues}, the eigenvalues of \eqref{eq:eigenvalue_triangle} are 
\begin{equation} \label{eq:eigenvalues}
    \lambda_{m,n} = \frac{16 \pi^2}{27}(m^2 - mn + n^2), \textrm{ with $m, n \in \Z$ satisfying} 
\end{equation}
\begin{itemize} 
\item[(A)] $m + n \equiv 0 \,\, (\text{mod } 3)$, 
\item[(B)] $m \neq 2n$, 
\item[(C)] $n \neq 2m$, 
\item[(D)] $m \neq -n$.
\end{itemize}
Although these conditions are stated in \cite{pinsky1980eigenvalues} and \cite{pinsky1985completeness}, there is no proof given that they are necessary and sufficient to guarantee that the associated $\lambda_{m,n}$ is an eigenvalue.  A formula is given for the associated eigenfunction, but it could happen that the function vanishes identically, or that it does not satisfy the boundary condition.  McCartin filled this gap by providing a beautiful pedagogical derivation of these expressions and showed how the conditions are necessary and sufficient \cite{mccartin2011laplacian}.  We have a slightly different proof that some readers may find more accessible, and since it takes a mere page, we include it for the benefit of readers.  

If $\lambda_{m,n}$ is an eigenvalue of $\Omega$, a corresponding eigenfunction is given by
\begin{equation} \label{eq:eigenfunction1}
    f_{m,n}(x,y) = \sum_{(m,n)} \pm e^{\frac{2\pi i}{3}(nx + \frac{(2m-n)y}{\sqrt{3}})}.
\end{equation}
Here, the sum goes through the six pairs in 
\begin{equation} \label{eq_six_pairs}
    (-n, m-n), (-n, -m), (n-m, -m), (n-m, n), (m,n), (m, m-n)
\end{equation}
from left to right and the sign alternates for each term. It is straightforward to verify that $(m,n)$ satisfies (A) -- (D) if and only if every pair above also satisfies (A) -- (D). Explicitly, the eigenfunction corresponding to these six pairs is 
\begin{align} \label{eq:eigenfunction2}
    \begin{split}
        f_{m,n}(x,y) &= e^{\frac{2\pi i}{3}((m-n)x - \frac{(m+n)y}{\sqrt{3}})} - e^{\frac{2\pi i}{3}(-mx + \frac{(m-2n)y}{\sqrt{3}})} + e^{\frac{2\pi i}{3}(-mx + \frac{(2n-m)y}{\sqrt{3}})} \\
        &- e^{\frac{2\pi i}{3}(nx + \frac{(n-2m)y}{\sqrt{3}})} + e^{\frac{2\pi i}{3}(nx + \frac{(2m-n)y}{\sqrt{3}})} - e^{\frac{2\pi i}{3}((m-n)x + \frac{(m+n)y}{\sqrt{3}})},
    \end{split}
\end{align}
or, equivalently,
\begin{equation}\label{eq:eigenfunction3}
    \begin{split} 
        f_{m,n}(x,y) = -2ie^{\frac{2\pi i}{3}(m-n)x} \sin\left(\frac{2\pi(m+n)y} {3\sqrt{3}}\right) 
        -2ie^{-\frac{2\pi i}{3}mx} \sin\left(\frac{2\pi(m-2n)y} {3\sqrt{3}}\right) \\
        -2ie^{\frac{2\pi i}{3}nx} \sin\left(\frac{2\pi(n-2m)y} {3\sqrt{3}}\right). 
    \end{split}
\end{equation}

We will now prove that $f_{m,n}$ is an eigenfunction of $\Omega$ with corresponding eigenvalue $\lambda_{m,n}$ 
if and only if (A), (B), (C), and (D) are satisfied. 
First, we note that 
\begin{equation*}
    (m-n)^2 + \frac{(m+n)^2}{3} = m^2 + \frac{(m-2n)^2}{3} = n^2 + \frac{(n-2m)^2}{3} = \frac{4m^2 - 4mn + 4n^2}{3},
\end{equation*}
from which it immediately follows that
\begin{equation*}
    \Delta f_{m,n} = -\frac{16\pi^2}{27}(m^2-mn+n^2)f_{m,n}.
\end{equation*}
Thus, $\Delta f_{m,n} + \lambda_{m,n}f_{m,n} = 0$ holds. Next, we examine when $f_{m,n}$ satisfies the boundary condition. At the boundary $y = 0$, it is clear from \eqref{eq:eigenfunction3} that $f_{m,n}$ vanishes. At $y = x\sqrt{3}$ we obtain
\begin{align*}
    f_{m,n}(x,x\sqrt{3}) &= e^{\frac{2\pi i}{3}(-2nx)} - e^{\frac{2\pi i}{3}(-2nx)} + e^{\frac{2\pi i}{3}((2n-2m)x)} \\ &- e^{\frac{2\pi i}{3}((2n-2m)x)} + e^{\frac{2\pi i}{3}(2mx)} - e^{\frac{2\pi i}{3}(2mx)} = 0.
\end{align*}
For $y = \sqrt{3}(1-x)$, we have
\begin{align*}
    f_{m,n}(x,\sqrt{3}(1-x)) 
    &= e^{\frac{2\pi i}{3}(-(m+n) + 2mx)} - e^{\frac{2\pi i}{3}((m-2n) + (2n-2m)x)} + e^{\frac{2\pi i}{3}((2n-m) - 2nx)} \\
    &- e^{\frac{2\pi i}{3}((n-2m) + 2mx)} + e^{\frac{2\pi i}{3}((2m-n) + (2n-2m)x)} - e^{\frac{2\pi i}{3}((m+n) - 2nx)},
\end{align*}
which we can write as
\begin{align*}
    f_{m,n}(x,\sqrt{3}(1-x)) &= e^{\frac{4\pi i}{3}mx}(e^{\frac{2\pi i}{3}(-m-n)} - e^{\frac{2\pi i}{3}(n-2m)}) \\
    &+ e^{\frac{4\pi i}{3}(n-m)x}(e^{\frac{2\pi i}{3}(2m-n)} - e^{\frac{2\pi i}{3}(m-2n)}) \\
    &+ e^{\frac{4\pi i}{3}(-nx)}(e^{\frac{2\pi i}{3}(2n-m)} - e^{\frac{2\pi i}{3}(m+n)}).
\end{align*}
This vanishes if and only if condition (A) holds. Indeed, if $m + n \equiv 0 \,\, (\text{mod } 3)$, then $m + n = 3k$ for some $k \in \Z$, and
\begin{align*}
    f_{m,n}(x,\sqrt{3}(1-x)) &= e^{\frac{4\pi i}{3}mx}(e^{\frac{2\pi i}{3}(-3k)} - e^{\frac{2\pi i}{3}(3k-3m)}) \\
    &+ e^{\frac{4\pi i}{3}(n-m)x}(e^{\frac{2\pi i}{3}(3m-3k)} - e^{\frac{2\pi i}{3}(3k-3n)}) \\
    &+ e^{\frac{4\pi i}{3}(-nx)}(e^{\frac{2\pi i}{3}(3n-3k)} - e^{\frac{2\pi i}{3}(3k)}) \\
    &= e^{\frac{4\pi i}{3}mx}(1-1) + e^{\frac{4\pi i}{3}(n-m)x}(1-1) + e^{\frac{4\pi i}{3}(-nx)}(1-1) = 0.
\end{align*}
If instead $m + n \equiv 1 \,\, (\text{mod } 3)$, so that $m + n = 3k + 1$, then it simplifies to
\begin{align*}
    f_{m,n}(x,\sqrt{3}(1-x)) = -i\sqrt{3}(e^{\frac{4\pi i }{3}mx} + e^{\frac{4\pi i }{3}(n-m)x} + e^{-\frac{4\pi i }{3}nx}),
\end{align*}
which is not identically zero. For example, at $x = 3/4$ this equals $-i\sqrt{3}((-1)^m + (-1)^{n-m} + (-1)^n) \neq 0$. Similarly, if $m + n \equiv 2 \,\, (\text{mod } 3)$, then $m + n = 3k - 1$ and
\begin{align*}
    f_{m,n}(x,\sqrt{3}(1-x)) = i\sqrt{3}(e^{\frac{4\pi i }{3}mx} + e^{\frac{4\pi i }{3}(n-m)x} + e^{-\frac{4\pi i }{3}nx}) \not\equiv 0.
\end{align*}

Next, we show that $f_{m,n} \equiv 0$ when either (B), (C), or (D) isn't satisfied. 
For this it suffices to check that $f_{m,2m}$, $f_{2n,n}$, and $f_{-n,n}$ are all identically zero.  To show this, we compute that the six pairs in each case respectively are 
\[ (-2m, -m), (-2m, -m), (m, -m), (m, 2m), (m, 2m), (m, -m);\]
\[ (-n, n), (-n, -2n), (-n, -2n), (-n, n) , (2n, n), (2n, n);\]
\[ (-n, -2n), (-n, n), (2n, n), (2n, n), (-n, n), (-n, -2n).\]
Due to the alternating signs in the definitions of $f_{m,2m}$, $f_{2n,n}$, and $f_{-n,n}$ it follows that they each vanish identically.  

Finally, we compute that 
\[ f_{m,n} (0, y) = -2 i \left[ \sin\left( \frac{2\pi y(m+n)}{3\sqrt 3}\right) + \sin \left( \frac{2\pi y(m-2n)}{3 \sqrt 3} \right) + \sin \left( \frac{2\pi y(n-2m)}{3 \sqrt 3} \right) \right].\]
Sines with different frequencies are linearly independent.  Consequently as long as $|m+n|$, $|m-2n|$ and $|n-2m|$ are not all equal, then $f_{m,n}(0,y)$ is not the zero function.  Since $f_{m,n}$ is a real analytic function on $\R^2$ it then follows that $f_{m,n}$ is also not the zero function.  We therefore compute 
\[ m^2 + n^2 + 2mn = m^2 + 4n^2 - 4mn = n^2 + 4m^2 - 4mn\]
\[ \iff n^2 + 2mn = 4n^2 - 4mn \iff 6mn = 3n^2 \iff n=0 \textrm{ or } 2m = n.\]
Since $2m\neq n$ by (C) this would require $n=0$ but then $m^2=4m^2 \implies m=0$ which violates (D). We have therewith shown that each orbit in \eqref{eq_six_pairs} satisfying (A)--(D) gives rise to a (nontrivial) Laplace eigenfunction that satisfies the Dirichlet boundary condition.  The fact that each orbit gives rise to a \em distinct \em eigenfunction, and that the collection of all of these functions constitutes an orthogonal base for $L^2$ on the triangle follows from \cite{pinsky1985completeness}.  

\subsection{The spectral zeta function and zeta-regularized determinant of an equilateral triangle} \label{sec:spectral_zeta_function_equilateral_triangle}
For the equilateral triangle with side length $\ell$, the spectral zeta function (corresponding to the Dirichlet boundary condition) is 
\begin{equation*}
    \zetaeq(s) = \sum_{\lambda_{m,n}} \frac{1}{\lambda_{m,n}^s} = \frac 1 6 \left(\frac{27\ell^2}{16\pi^2}\right)^s\sum_{\lambda_{m,n}} \frac{1}{(m^2-mn+n^2)^s}.
\end{equation*}
We sum according to \eqref{eq:eigenvalues}.  Each eigenvalue occurs six times its actual multiplicity, hence we divide by $6$ to correctly account for multiplicities.  Our first result is an expression for this spectral zeta function that we have not encountered in the literature. 

\begin{proposition} \label{prop:zeta_eqtri} 
The spectral zeta function for an equilateral triangle with side lengths equal to $\ell$ and the Dirichlet boundary condition is equivalently
\begin{align*}   \zetaeq(s) 
    &= \frac 1 6 \left(\frac{3\ell}{4\pi}\right)^{2s} \bigg[ -4\zeta_R(2s) + \frac{2^{2s}\sqrt{\pi}\zeta_R(2s-1)\Gamma(s-1/2)}{\Gamma(s)3^{s-1/2}} \\ 
    &+ \frac{4\pi^s2^{s-1/2}}{\Gamma(s)3^{s/2-1/4}}\sum_{n=1}^\infty n^{s-1/2}\sum_{d|n}d^{1-2s}(-1)^n \int_0^\infty x^{s-3/2}e^{-\pi n\sqrt{3}(x+x^{-1})/2} dx \bigg], \\
    \zetaeq(s) &= \frac{1}{6}\left(\frac{3\ell}{4}\right)^{2s}\left[G_\nabla(s) - \frac{6}{\pi^{2s}}\zeta_R(2s)\right].
\end{align*}
Here,
\begin{equation}
    G_\nabla(s) = \sum_{m \in \Z}\sum_{k \in \Z}\frac{1}{\pi^{2s}|m+kz|^{2s}}, \,\, z = \frac{-3 + i\sqrt{3}}{2}. \label{eq:Gtri}
\end{equation}

\end{proposition}

\begin{proof}
To calculate the sum defining the spectral zeta function we begin by restricting the sum to nonzero pairs  $(m,n) = (m,3k-m)$ satisfying (A):
\begin{align*}
    &\sum_{m \in \Z}\sum_{k \in \Z} \frac{1}{(m^2 - m(3k-m) + (3k-m)^2)^s} = \sum_{m \in \Z}\sum_{k \in \Z} \frac{1}{(3m^2 - 9km + 9k^2)^s} \\ 
    &= \frac{1}{3^s} \sum_{m \in \Z}\sum_{k \in \Z} \frac{1}{(m^2 - 3km + 3k^2)^s}. 
\end{align*}
The notation $m \in \Z$, $k \in \Z$ is as in \eqref{eq:abuse}. By \cite[p. 87]{selberg1967epstein}, 
we have
\begin{align*}
    \sum_{m \in \Z}\sum_{k \in \Z} \frac{1}{(m^2-3km+3k^2)^s} = 2\zeta_R(2s) + \frac{2^{2s}\sqrt{\pi}\zeta_R(2s-1)\Gamma(s-1/2)}{\Gamma(s)3^{s-1/2}} + Q(s), \\
    Q(s) = \frac{4\pi^s2^{s-1/2}}{\Gamma(s)3^{s/2-1/4}}\sum_{n=1}^\infty n^{s-1/2}\sum_{d|n} d^{1-2s}(-1)^n \int_0^\infty x^{s-3/2}e^{-\pi n\sqrt{3}(x+x^{-1})/2} dx. 
\end{align*}

When we sum over the pairs $(m,n)$ with $m = 2n$, we get
\begin{align*}
    \sum_{n \in \Z, n \neq 0} \frac{1}{((2n)^2 - 2n\cdot n + n^2)^s} = \sum_{n \in \Z, n \neq 0} \frac{1}{(3n^2)^s} = \frac{1}{3^s}\sum_{n \in \Z, n \neq 0} \frac{1}{n^{2s}} = \frac{2}{3^s}\sum_{n = 1}^\infty \frac{1}{n^{2s}} = \frac{2}{3^s} \zeta_R(2s).
\end{align*}
Of course we get the same result when we sum over the pairs $(m,n)$ with $n = 2m$ and $m = -n$. Thus, recalling the factor of $\frac 1 6$, we have 
\begin{align*}
    \zetaeq(s) &= \frac 1 6 \left(\frac{27\ell^2}{16\pi^2}\right)^s \bigg[ \frac{1}{3^s} \sum_{m \in \Z}\sum_{k \in \Z} \frac{1}{(m^2 - 3km + 3k^2)^s} - 3\cdot \frac{2}{3^s}\zeta_R(2s) \bigg]  \\
    &=\frac 1 6 \left(\frac{3\ell}{4\pi}\right)^{2s} \bigg[ -4\zeta_R(2s) + \frac{2^{2s}\sqrt{\pi}\zeta_R(2s-1)\Gamma(s-1/2)}{\Gamma(s)3^{s-1/2}} \\ 
    &+ \frac{4\pi^s2^{s-1/2}}{\Gamma(s)3^{s/2-1/4}}\sum_{n=1}^\infty n^{s-1/2}\sum_{d|n}d^{1-2s}(-1)^n \int_0^\infty x^{s-3/2}e^{-\pi n\sqrt{3}(x+x^{-1})/2} dx \bigg].
\end{align*}
To obtain the second expression for $\zetaeq(s)$, simply note that 
\begin{equation*}
    \sum_{m \in \Z}\sum_{k \in \Z} \frac{1}{(m^2 - 3km + 3k^2)^s} = \sum_{m \in \Z}\sum_{k \in \Z} \frac{1}{|m+kz|^{2s}} = G_\nabla(s), \,\, z = \frac{-3 + i\sqrt{3}}{2},
\end{equation*}
and therewith 
\begin{equation*}
    \zetaeq(s) = \frac{1}{6}\left(\frac{3\ell}{4}\right)^{2s}\left[G_\nabla(s) - \frac{6}{\pi^{2s}}\zeta_R(2s)\right].
\end{equation*}
\end{proof} 

These expressions allow one to meromorphically extend the spectral zeta function to the complex plan and evaluate it at points of interest.  In particular, as a corollary we obtain two expressions for $\zetaeq'(0)$ which can be used to calculate the zeta-regularized determinant, $e^{-\zetaeq'(0)}$. 

\begin{corollary} \label{cor:zetadet_eqtri} 
The derivative of the spectral zeta function of an equilateral triangle with sides of length $\ell$ and the Dirichlet boundary condition is equivalently given by the expressions 
\begin{align*}
    \zetaeq'(0) &= \frac 2 3 \log\left(\frac{3 \ell }{2}\right) + \frac{\pi\sqrt{3}}{36} + \frac 2 3 \sum_{n=1}^\infty \frac{(-1)^n}{n e^{\pi n \sqrt{3}}}\sum_{d|n} d, \\ 
    \zetaeq'(0) &= \frac 2 3\log\left(\frac{3 \ell}{2|\eta(z)|}\right), \quad z = \frac{-3 + i\sqrt{3}}{2}. 
\end{align*}
Above, $\eta$ is the Dedekind eta function. 
\end{corollary} 

\begin{proof} 
We have
\begin{align*}
    \zetaeq'(s) &= \frac{1}{3}\log\left(\frac{3 \ell}{4\pi}\right)\left(\frac{3 \ell}{4\pi}\right)^{2s} \bigg[ -4\zeta_R(2s) + \frac{2^{2s}\sqrt{\pi}\zeta_R(2s-1)\Gamma(s-1/2)}{\Gamma(s)3^{2s-1/2}} \\ 
    &+ \frac{4\pi^s2^{s-1/2}}{\Gamma(s)3^{s/2-1/4}}\sum_{n=1}^\infty n^{s-1/2}\sum_{d|n}d^{1-2s}(-1)^n \int_0^\infty x^{s-3/2}e^{-\pi n\sqrt{3}(x+x^{-1})/2} dx \bigg] \\
    &- \frac{1}{6}\left(\frac{3 \ell}{4\pi}\right)^{2s} 8\zeta_R'(2s) +  \frac{\sqrt{3\pi}}{\Gamma(s)^2}\left(\frac{3 \ell ^2}{4\pi^2}\right)^s \biggl[\log\left(\frac{4}{3}\right)\zeta_R(2s-1)\Gamma(s)\Gamma(s-1/2) \\ &+ 2\zeta_R'(2s-1)\Gamma(s)\Gamma(s-1/2) - \zeta_R(2s-1)\Gamma'(s)\Gamma(s-1/2) + \zeta_R(2s-1)\Gamma(s)\Gamma'(s-1/2)\biggr] \\
    &+ \frac{1}{6}\frac{d}{ds} \bigg[\frac{4\pi^s2^{s-1/2}}{\Gamma(s)3^{s/2-1/4}}\sum_{n=1}^\infty n^{s-1/2}\sum_{d|n}d^{1-2s}(-1)^n \int_0^\infty x^{s-3/2}e^{-\pi n\sqrt{3}(x+x^{-1})/2} dx\bigg],
\end{align*}
where
\begin{align*}
    &\frac{d}{ds} \bigg[\frac{4\pi^s2^{s-1/2}}{\Gamma(s)3^{s/2-1/4}}\sum_{n=1}^\infty n^{s-1/2}\sum_{d|n}d^{1-2s}(-1)^n \int_0^\infty x^{s-3/2}e^{-\pi n\sqrt{3}(x+x^{-1})/2} dx\bigg] \\
    = &-\frac{\Gamma'(s)}{\Gamma(s)^2}\frac{4\pi^s2^{s-1/2}}{3^{s/2-1/4}}\sum_{n=1}^\infty n^{s-1/2}\sum_{d|n}d^{1-2s}(-1)^n \int_0^\infty x^{s-3/2}e^{-\pi n\sqrt{3}(x+x^{-1})/2} dx \\
    &+ \frac{1}{\Gamma(s)}\frac{d}{ds} \bigg[\frac{4\pi^s2^{s-1/2}}{3^{s/2-1/4}}\sum_{n=1}^\infty n^{s-1/2}\sum_{d|n}d^{1-2s}(-1)^n \int_0^\infty x^{s-3/2}e^{-\pi n\sqrt{3}(x+x^{-1})/2} dx\bigg].
\end{align*}

By Lemma \ref{lemma_vanishes} which is stated and proved in the appendix, the last term vanishes as $s \to 0$, and we thus obtain 
\begin{align*}
    \zetaeq'(0) &= \frac{1}{3}\log\left(\frac{3 \ell}{4\pi}\right)(-4\zeta_R(0)) - \frac{4}{3}\zeta_R'(0) - \frac{\sqrt{3\pi}\zeta_R(-1)\Gamma(-1/2)\Gamma'(0)}{6\Gamma(0)^2} \\
    &- \frac{2\sqrt[4]{3}}{3\sqrt{2}}\frac{\Gamma'(0)}{\Gamma(0)^2}\sum_{n=1}^\infty n^{-1/2}\sum_{d|n}d(-1)^n \int_0^\infty x^{-3/2}e^{-\pi n\sqrt{3}(x+x^{-1})/2} dx. 
\end{align*}
By \eqref{eq:gamma_function_simple_pole}, \eqref{eq:miracle_integral}, and \eqref{eq:values_Riemann_zeta_Gamma}, this simplifies to
\begin{equation*}
    \zetaeq'(0) = \frac 2 3 \log\left(\frac{3\ell}{2}\right) + \frac{\pi\sqrt{3}}{36} + \frac 2 3 \sum_{n=1}^\infty \frac{(-1)^n}{n e^{\pi n \sqrt{3}}}\sum_{d|n} d.
\end{equation*}
To obtain the second expression for $\zetaeq'(0)$, we differentiate the second expression for $\zetaeq(s)$. This gives 
\begin{align*}
    \zetaeq'(s) &= \frac{1}{3}\log\left(\frac{3\ell}{4}\right)\left(\frac{3\ell}{4}\right)^{2s}\left[G_\nabla(s) - \frac{6}{\pi^{2s}}\zeta_R(2s)\right] \\
    &+ \frac{1}{6}\left(\frac{3\ell}{4}\right)^{2s}\left[G_\nabla'(s) + \frac{12\log(\pi)}{\pi^{2s}}\zeta_R(2s) - \frac{12}{\pi^{2s}}\zeta_R'(2s)\right],
\end{align*}
which at $s = 0$ becomes
\begin{equation*}
    \zetaeq'(0) = \frac{1}{3}\log\left(\frac{3\ell}{4}\right)(G_\nabla(0) + 3) + \frac{1}{6}G_\nabla'(0) + \log(2).
\end{equation*}
We have by \cite[p. 204-205]{osgood1988extremals} (see also \cite[p. 1830-1831]{aldana2018polyakov}),
\begin{equation*}
    G_\nabla(0) = -1, \,\, G_\nabla'(0) = -\frac{1}{12}\log\left((2\pi)^{24}\frac{(\eta(z)\bar{\eta}(z))^{24}}{\pi^{24}}\right).
\end{equation*}
This yields
\begin{equation*}
  \zetaeq'(0) = \frac{2}{3}\log\left(\frac{3\ell}{2|\eta(z)|}\right).
\end{equation*}
If we write out $|\eta(z)|$ explicitly, we get
\begin{align*}
    |\eta(z)| &= e^{-\frac{\pi\sqrt{3}}{24}} \prod_{n=1}^\infty (1+(-1)^{n+1}e^{-\pi n \sqrt{3}}), \\
    \log(|\eta(z)|) &= -\frac{\pi\sqrt{3}}{24} + \sum_{n=1}^\infty \log(1+(-1)^{n+1}e^{-\pi n \sqrt{3}}),
\end{align*}
hence
\begin{equation*}
 \zetaeq'(0) = \frac{2}{3}\log\left(\frac{3\ell}{2}\right) + \frac{\pi\sqrt{3}}{36} - \frac{2}{3}\sum_{n=1}^\infty \log(1+(-1)^{n+1}e^{-\pi n \sqrt{3}}).
\end{equation*}
By doing a Taylor expansion of the logarithm, we see that this can be written as
\begin{equation*}
  \zetaeq'(0) = \frac{2}{3}\log\left(\frac{3\ell}{2}\right) + \frac{\pi\sqrt{3}}{36} + \frac{2}{3}\sum_{n=1}^\infty \frac{(-1)^n}{ne^{\pi n \sqrt{3}}}\sum_{d|n} d.
\end{equation*}
As a consequence, we have explicitly verified that our two expressions for $\zetaeq'(0)$ agree. 
\end{proof} 

\subsection{The heat trace of equilateral triangles and an alternative expression for the eigenvalues} \label{ss:eqtri_heat}
There is another common expression for the eigenvalues of an equilateral triangle in the literature (see e.g. \cite{mccartin2003eigenstructure}), namely
\begin{equation*}
    \lambda_{m,n} = \frac{4\pi^2}{27r^2}(m^2+mn+n^2), \,\, m,n \geq 1.
\end{equation*}
Here, $r$ is the radius of the inscribed circle of the triangle. If the triangle has side lengths each equal to $\ell$, then $r = \ell/\sqrt{12}$. Therefore this expression for the eigenvalues in case $\ell = 1$ is simply 
\begin{equation}
    \lambda_{m,n} = \frac{16\pi^2}{9}(m^2+mn+n^2), \,\, m,n \geq 1.
\label{eq:eqtri_ev_norm} \end{equation}
This is different from the expression \eqref{eq:eigenvalues} given by Pinsky \cite{pinsky1980eigenvalues, pinsky1985completeness} and Lam\'e \cite{lame1833memoire, lame1852leccons, lame1861leccons}.  It may appear more simple for computations because it no longer involves the conditions (A)--(D) on the integers, and their range is $\N$ rather than $\Z$.  At the same time, the connection to eigenfunctions is obfuscated, as is the multiplicity of the eigenvalues.  By \cite{pinsky1980eigenvalues, pinsky1985completeness}, expressing the eigenvalues as \eqref{eq:eigenvalues}, we know that there are six pairs that correspond to one linearly independent eigenfunction, given by the six pairs in \eqref{eq_six_pairs}.  Each distinct orbit gives rise to a distinct, linearly independent eigenfunction.  Hence to calculate spectral invariants like the spectral zeta function or the heat trace, it suffices to sum over all integers $(m,n)$ satisfying (A)--(D) and then divide by six.  It is not at all clear how to account for multiplicities using the expression \eqref{eq:eqtri_ev_norm}.  Here we use the heat trace to show how to account for the multiplicities correctly if one wishes to use \eqref{eq:eqtri_ev_norm} to compute eigenvalues and spectral invariants of equilateral triangles.   

\begin{proposition} \label{prop:2heat_traces}
The heat trace for the equilateral triangle with side length $\ell$ and the Dirichlet boundary condition is equivalently given by the expressions 
\begin{align}
    \heateqd(t) &=  \frac{\Theta_3(q^3)\Theta_3(q^9) + \Theta_2(q^3)\Theta_2(q^9) - 3\Theta_3(q^3) + 2}{6}, \nn \\
    \Theta_2(q) &= \sum_{n \in \Z} q^{(n+1/2)^2}, \quad 
    \Theta_3(q) = \sum_{n \in \Z} q^{n^2}, \quad q= e^{-\frac{16 \pi^2}{27 \ell ^2}t}, \nn \\
    \heateqd(t) &= \sum_{m = 1}^\infty \sum_{n = 1}^\infty e^{-\frac{16\pi^2}{9 \ell^2}(m^2+mn+n^2)t}. \label{eq:verhoeven_heat}
\end{align}

As a consequence, the eigenvalues are the values
\beq \frac{16 \pi^2}{9 \ell^2}(m^2+mn+n^2), \quad m, n \geq 1. \label{eq:ev_mn_geq1} \eeq 
For each pair $(m,n)$ with $m,n \geq 1$ there is exactly one orbit of the form \eqref{eq_six_pairs} where each of the six pairs in the orbit satisfies conditions (A), (B), (C), and (D). 
\end{proposition}

\begin{proof}
For the equilateral triangle with side length $\ell$, the heat trace with the Dirichlet boundary condition is 
\begin{equation*}
    \heateqd(t) = \frac 1 6 \sum_{\lambda_{m,n}} e^{-\frac{16\pi^2}{27 \ell^2}(m^2 -mn + n^2)t}.
\end{equation*}
The sum goes through all pairs $(m,n) \in \Z^2$ satisfying (A), (B), (C), and (D). To compensate for the fact that the six pairs in \eqref{eq_six_pairs} all correspond to the same eigenvalue, we have divided by 6. 

Then, for $q = e^{-\frac{16\pi^2}{27\ell^2}t}$, we get when summing over \textit{all} integer pairs 
\begin{align} \label{eq:heat_trace_all_integers}
\begin{split}
    &\sum_{m \in \Z}\sum_{n \in \Z} q^{m^2 - mn + n^2} = \sum_{m \in \Z}\sum_{n \in \Z} q^{(m-n/2)^2} q^{3n^2/4} = \sum_{n \in \Z} q^{3n^2/4} \sum_{m \in \Z} q^{(m-n/2)^2} \\
    = & \sum_{n \in 2\Z} q^{3n^2/4} \sum_{m \in \Z} q^{m^2} + \sum_{n \in 2\Z+1} q^{3n^2/4} \sum_{m \in \Z} q^{(m+1/2)^2} \\
    = & \Theta_3(q) \sum_{n \in 2\Z} q^{3n^2/4} + \Theta_2(q) \sum_{n \in 2\Z+1} q^{3n^2/4} \\
    = & \Theta_3(q) \sum_{k \in \Z} q^{3(2k)^2/4} + \Theta_2(q) \sum_{k \in \Z} q^{3(2k+1)^2/4} \\ 
    = & \Theta_3(q) \sum_{k \in \Z} q^{3k^2} + \Theta_2(q) \sum_{k \in \Z} q^{3(k+1/2)^2} = \Theta_3(q)\Theta_3(q^3) + \Theta_2(q)\Theta_2(q^3).
\end{split}
\end{align}
Now let us only sum over the pairs $(m,n)$ which satisfy (A). We write $n = 3k-m$ and obtain
\begin{align*}
    h_1(t) &= \sum_{k \in \Z}\sum_{m \in \Z} q^{m^2 -m(3k-m) + (3k-m)^2} = \sum_{k \in \Z} q^{9k^2/4}\sum_{m \in \Z} q^{3(m-3k/2)^2} \\ & = \Theta_3(q^3)\Theta_3(q^9) + \Theta_2(q^3)\Theta_2(q^9)
\end{align*}
by computing in a similar way. To obtain $\heateqd(t)$, we must subtract the contribution from the pairs $(m,n)$ with $m = 2n$ or $n = 2m$ or $m = -n$:
\begin{align*}
    & h_2(t) = \sum_{n \in \Z} q^{(2n)^2 - 2n\cdot n + n^2} = \sum_{n \in \Z} q^{3n^2} = \Theta_3(q^3), \\
    & h_3(t) = \sum_{m \in \Z} q^{m^2 - m\cdot 2m + (2m)^2} = \sum_{m \in \Z} q^{3m^2} = \Theta_3(q^3), \\
    & h_4(t) = \sum_{n \in \Z} q^{(-n)^2 - (-n)n + n^2} = \sum_{n \in \Z} q^{3n^2} = \Theta_3(q^3).
\end{align*}
Thus, we get
\begin{align*}
    \heateqd(t) &= \frac 1 6 \sum_{\lambda_{m,n}} e^{-\lambda_{m,n}t} = \frac{h_1(t) - h_2(t) - h_3(t) - h_4(t) + 2}{6} \\ &= \frac{\Theta_3(q^3)\Theta_3(q^9) + \Theta_2(q^3)\Theta_2(q^9) - 3\Theta_3(q^3) + 2}{6}.
\end{align*}
The plus 2 appears because when we subtract all pairs $(m,n)$ with $m = 2n$ or $n = 2m$ or $m = -n$, we subtract the contribution from $(0,0)$ three times.  

Now, \cite{verhoevencan} starts with \eqref{eq:verhoeven_heat} and computes 
\begin{align*}
    \heateqd(t) &= \sum_{m = 1}^\infty \sum_{n = 1}^\infty e^{-\frac{16\pi^2}{9\ell^2}(m^2+mn+n^2)t} \\ &= 
    \frac{1}{6}\left[ \sum_{m \in \Z}\sum_{n \in \Z}e^{-\frac{16\pi^2}{9\ell^2}(m^2+mn+n^2)t} - 3\sum_{m \in \Z}e^{-\frac{16\pi^2}{9\ell^2}m^2t} + 2 \right] \\ &= \frac{1}{6}\left[ \sum_{m \in \Z}\sum_{n \in \Z}q^{3(m^2+mn+n^2)} - 3\sum_{m \in \Z}q^{3m^2} + 2 \right].
\end{align*}
By \eqref{eq:heat_trace_all_integers} this is 
\begin{equation*}
    \frac{\Theta_3(q^3)\Theta_3(q^9) + \Theta_2(q^3)\Theta_2(q^9) - 3\Theta_3(q^3) + 2}{6}.
\end{equation*}
In particular, the two expressions for the heat trace are identical. This proves the equivalence between the two expressions for the eigenvalues of equilateral triangles. 
\end{proof} 
\begin{remark} Observe that in the expression \eqref{eq:ev_mn_geq1} it is not at all obvious that each eigenvalue occurs precisely \em six times \em but this is indeed the case.  It is also interesting to recall the observation made in \cite[Corollary 2.2.6]{verhoevencan}:  there is no bounded domain in $\R^2$ whose spectrum has the form 
\[ c(m^2 - mn + n^2), \quad m,n \in \Z.\]
In particular, the restrictions (A) -- (D) on the pairs $(m,n) \in \Z \times \Z$ are essential to the correct expression for the eigenvalues of the equilateral triangle.  To the best of our knowledge, ours is the first proof of the equivalence of these two expressions for the eigenvalues of the equilateral triangle.
\end{remark}

\subsection{Short time asymptotic expansion of the heat trace} \label{ss:eq_tri_ht_asy}
Here we obtain further terms in the short time asymptotic expansion of the heat trace.  We note that \cite{verhoevencan} obtained a related formula but rather than extracting further terms explicitly, collected all terms that vanish as $t \to 0$ into one big-O term using the asymptotic behavior of Jacobi theta functions. 

\begin{theorem}\label{th:eqtri_asy_htrace} The Dirichlet heat trace for an equilateral triangle with sides of length $\ell$ has the asymptotic expansion as $t \to 0$ 
\begin{equation*}
    \heateqd (t) = \frac{\ell^2\sqrt{3}}{16\pi t} - \frac{3\ell}{8\sqrt{\pi t}} + \frac{1}{3} - \frac{3\ell}{4\sqrt{\pi t}} e^{- \frac{9\ell^2}{16t}} + \frac{\ell^2\sqrt{3}}{8\pi t} e^{-\frac{3\ell^2}{4t}}  + \mathcal O (t^{-1} e^{-\frac{9\ell^2}{4t}} ).
\end{equation*}
The Neumann heat trace has the asymptotic expansion as $t \to 0$ 
\begin{equation*}
    \heateqn (t) = \frac{\ell^2\sqrt{3}}{16\pi t} + \frac{3\ell}{8\sqrt{\pi t}} + \frac{1}{3} + \frac{3\ell}{4\sqrt{\pi t}} e^{- \frac{9\ell^2}{16t}}  + \frac{\ell^2\sqrt{3}}{8\pi t} e^{-\frac{3\ell^2}{4t}}  + \mathcal O (t^{-1} e^{-\frac{9\ell^2}{4t}} ).
\end{equation*}
The remainders are sharp. 
\end{theorem}

\begin{proof} 
By Proposition \ref{prop:2heat_traces} the Dirichlet heat trace is 
\begin{align*}
    \heateqd(t) &= \frac{\Theta_3(q^3)\Theta_3(q^9) + \Theta_2(q^3)\Theta_2(q^9) - 3\Theta_3(q^3) + 2}{6} 
    = \frac{1}{6}\bigg[\sum_{m \in \Z}e^{-\frac{16\pi^2}{9\ell^2}m^2t}\sum_{n \in \Z}e^{-\frac{16\pi^2}{3\ell^2}n^2t} \\ &+ \sum_{m \in \Z}e^{-\frac{16\pi^2}{9\ell^2}(m+\frac{1}{2})^2t}\sum_{n \in \Z}e^{-\frac{16\pi^2}{3\ell^2}(n+\frac{1}{2})^2t} - 3\sum_{m \in \Z}e^{-\frac{16\pi^2}{9\ell^2}m^2t} + 2\bigg].
\end{align*}
By Poisson's summation formula,
\begin{align*}
    \sum_{m \in \Z}e^{-\frac{16\pi^2}{9\ell^2}m^2t} &= \frac{3\ell}{4\sqrt{\pi t}}\sum_{m \in \Z}e^{-\frac{9\ell^2m^2}{16t}}, \\
    \sum_{n \in \Z}e^{-\frac{16\pi^2}{3\ell^2}n^2t} &= \frac{\ell}{4}\sqrt{\frac{3}{\pi t}}\sum_{n \in \Z}e^{-\frac{3\ell^2n^2}{16t}}, \\
    \sum_{m \in \Z}e^{-\frac{16\pi^2}{9\ell^2}(m+\frac{1}{2})^2t} &= \frac{3\ell}{4\sqrt{\pi t}}\sum_{m \in \Z}(-1)^m e^{-\frac{9\ell^2m^2}{16t}}, \\
    \sum_{n \in \Z}e^{-\frac{16\pi^2}{3\ell^2}(n+\frac{1}{2})^2t} &= \frac{\ell}{4}\sqrt{\frac{3}{\pi t}}\sum_{n \in \Z}(-1)^n e^{-\frac{3\ell^2n^2}{16t}}.
\end{align*}
This gives
\begin{align*}
    \heateqd(t) &= \frac{1}{6}\bigg[\frac{3\ell^2\sqrt{3}}{16\pi t}\sum_{m \in \Z}e^{-\frac{9\ell^2m^2}{16t}}\sum_{n \in \Z}e^{-\frac{3\ell^2n^2}{16t}} + \frac{3\ell^2\sqrt{3}}{16\pi t}\sum_{m \in \Z}(-1)^me^{-\frac{9\ell^2m^2}{16t}}\sum_{n \in \Z}(-1)^ne^{-\frac{3\ell^2n^2}{16t}} \\
    &- \frac{9\ell}{4\sqrt{\pi t}}\sum_{m \in \Z} e^{-\frac{9\ell^2m^2}{16t}} + 2\bigg] \\
    &= \frac{1}{6}\bigg[\frac{3\ell^2\sqrt{3}}{16\pi t}\left(1 + 2\sum_{m = 1}^\infty e^{-\frac{9\ell^2m^2}{16t}}\right)\left(1 + 2\sum_{n = 1}^\infty e^{-\frac{3\ell^2n^2}{16t}}\right) \\
    &+ \frac{3\ell^2\sqrt{3}}{16\pi t}\left(1 + 2\sum_{m = 1}^\infty (-1)^me^{-\frac{9\ell^2m^2}{16t}}\right)\left(1 + 2\sum_{n = 1}^\infty (-1)^ne^{-\frac{3\ell^2n^2}{16t}}\right) \\
    &- \frac{9\ell}{4\sqrt{\pi t}}\left(1 + 2\sum_{m = 1}^\infty e^{-\frac{9\ell^2m^2}{16t}}\right) + 2\biggr] \\
    &= \frac{\ell^2\sqrt{3}}{16\pi t} - \frac{3\ell}{8\sqrt{\pi t}} + \frac{1}{3} + \frac{\ell^2\sqrt{3}}{16\pi t}\sum_{m=1}^\infty e^{-\frac{9\ell^2m^2}{16t}} + \frac{\ell^2\sqrt{3}}{16\pi t}\sum_{n=1}^\infty e^{-\frac{3\ell^2n^2}{16t}} + \frac{\ell^2\sqrt{3}}{8\pi t}\sum_{m=1}^\infty\sum_{n=1}^\infty e^{-\frac{3\ell^2(3m^2+n^2)}{16t}} \\
    &+ \frac{\ell^2\sqrt{3}}{16\pi t}\sum_{m=1}^\infty (-1)^me^{-\frac{9\ell^2m^2}{16t}} + \frac{\ell^2\sqrt{3}}{16\pi t}\sum_{n=1}^\infty (-1)^ne^{-\frac{3\ell^2n^2}{16t}} + \frac{\ell^2\sqrt{3}}{8\pi t}\sum_{m=1}^\infty\sum_{n=1}^\infty (-1)^{m+n}e^{-\frac{3\ell^2(3m^2+n^2)}{16t}} \\
    &- \frac{3\ell}{4\sqrt{\pi t}}\sum_{m=1}^\infty e^{-\frac{9\ell^2m^2}{16t}}.
\end{align*}
Since
\begin{align*}
    &\sum_{m=1}^\infty e^{-\frac{9\ell^2m^2}{16t}} + \sum_{m=1}^\infty (-1)^m e^{-\frac{9\ell^2m^2}{16t}} = 2\sum_{m=1}^\infty e^{-\frac{9\ell^2m^2}{4t}}, \\ 
    &\sum_{n=1}^\infty e^{-\frac{3\ell^2n^2}{16t}} + \sum_{n=1}^\infty (-1)^n e^{-\frac{3\ell^2n^2}{16t}} = 2\sum_{n=1}^\infty e^{-\frac{3\ell^2n^2}{4t}}, \\
    &\sum_{m=1}^\infty\sum_{n=1}^\infty e^{-\frac{3\ell^2(3m^2+n^2)}{16t}} + \sum_{m=1}^\infty\sum_{n=1}^\infty (-1)^{m+n}e^{-\frac{3\ell^2(3m^2+n^2)}{16t}} \\
    = 2&\sum_{m=1}^\infty\sum_{n=1}^\infty e^{-\frac{3\ell^2(3m^2+n^2)}{4t}} + 2\sum_{m=1}^\infty\sum_{n=1}^\infty e^{-\frac{3\ell^2(3(2m-1)^2+(2n-1)^2)}{16t}},
\end{align*}
this simplifies to
\begin{align*}
    \heateqd(t) &= \frac{\ell^2\sqrt{3}}{16\pi t} - \frac{3\ell}{8\sqrt{\pi t}} + \frac{1}{3} - \frac{3\ell}{4\sqrt{\pi t}}\sum_{m=1}^\infty e^{-\frac{9\ell^2m^2}{16t}} + \frac{\ell^2\sqrt{3}}{8\pi t}\sum_{n=1}^\infty e^{-\frac{3\ell^2n^2}{4t}} + \frac{\ell^2\sqrt{3}}{8\pi t}\sum_{m=1}^\infty e^{-\frac{9\ell^2m^2}{4t}} \\
    &+ \frac{\ell^2\sqrt{3}}{4\pi t}\sum_{m=1}^\infty\sum_{n=1}^\infty e^{-\frac{3\ell^2(3m^2+n^2)}{4t}} + \frac{\ell^2\sqrt{3}}{4\pi t}\sum_{m=1}^\infty\sum_{n=1}^\infty e^{-\frac{3\ell^2(3(2m-1)^2+(2n-1)^2)}{16t}}.
\end{align*}
The eigenvalues of the same equilateral triangle with Neumann boundary conditions are
\begin{equation*}
    \lambda_{m,n} = \frac{16\pi^2}{9\ell^2}(m^2+mn+n^2), \,\, m, n \geq 0,
\end{equation*}
so the heat trace becomes
\begin{align*}
    \heateqn(t) &= \sum_{m=0}^\infty\sum_{n=0}^\infty e^{-\frac{16\pi^2}{9\ell^2}(m^2+mn+n^2)t} = \sum_{n=0}^\infty e^{-\frac{16\pi^2}{9\ell^2}n^2t} + \sum_{m=1}^\infty\sum_{n=0}^\infty e^{-\frac{16\pi^2}{9\ell^2}(m^2+mn+n^2)t} \\
    &= 1 + \sum_{n=1}^\infty e^{-\frac{16\pi^2}{9\ell^2}n^2t} + \sum_{m=1}^\infty e^{-\frac{16\pi^2}{9\ell^2}m^2t} + \sum_{m=1}^\infty\sum_{n=1}^\infty e^{-\frac{16\pi^2}{9\ell^2}(m^2+mn+n^2)t} \\
    &= 1 + 2\sum_{m=1}^\infty e^{-\frac{16\pi^2}{9\ell^2}m^2t} + \heateqd(t).
\end{align*}
Since
\begin{equation*}
    \sum_{m=1}^\infty e^{-\frac{16\pi^2}{9\ell^2}m^2t} = \frac{1}{2}\left(\frac{3\ell}{4\sqrt{\pi t}} - 1\right) + \frac{3\ell}{4\sqrt{\pi t}} \sum_{m=1}^\infty e^{-\frac{9\ell^2m^2}{16t}},
\end{equation*}
we get
\begin{align*}
    \heateqn(t) &= \frac{\ell^2\sqrt{3}}{16\pi t} + \frac{3\ell}{8\sqrt{\pi t}} + \frac{1}{3} + \frac{3\ell}{4\sqrt{\pi t}}\sum_{m=1}^\infty e^{-\frac{9\ell^2m^2}{16t}} + \frac{\ell^2\sqrt{3}}{8\pi t}\sum_{n=1}^\infty e^{-\frac{3\ell^2n^2}{4t}} + \frac{\ell^2\sqrt{3}}{8\pi t}\sum_{m=1}^\infty e^{-\frac{9\ell^2m^2}{4t}} \\
    &+ \frac{\ell^2\sqrt{3}}{4\pi t}\sum_{m=1}^\infty\sum_{n=1}^\infty e^{-\frac{3\ell^2(3m^2+n^2)}{4t}} + \frac{\ell^2\sqrt{3}}{4\pi t}\sum_{m=1}^\infty\sum_{n=1}^\infty e^{-\frac{3\ell^2(3(2m-1)^2+(2n-1)^2)}{16t}}.
\end{align*}

\end{proof} 
It follows from Theorem \ref{th:eqtri_asy_htrace} that 
\begin{align*}
    \heateqd(t) &= \frac{\ell^2\sqrt{3}}{16\pi t} - \frac{3\ell}{8\sqrt{\pi t}} + \frac{1}{3} + \mathcal{O}(e^{-\frac{9\ell^2 - \epsilon}{16t}}), \,\, t \to 0, \\
    \heateqn(t) &= \frac{\ell^2\sqrt{3}}{16\pi t} + \frac{3\ell}{8\sqrt{\pi t}} + \frac{1}{3} + \mathcal{O}(e^{-\frac{9\ell^2 - \epsilon}{16t}}), \,\, t \to 0,
\end{align*}
for any $\epsilon > 0$. Consequently, the infimum over all $c$ such that the remainder is $\mathcal O (e^{-c/t})$ is $\frac{9 \ell^2}{16}$.  By \cite[p. 43]{durso1988inverse}, this is the square of half the length of the shortest closed geodesic in the equilateral triangle of side length $\ell$.

\section{Spectral invariants of isosceles right triangles} \label{sec:isosceles_right}
By \cite{aronovitch2012nodal}, the eigenvalues of an isosceles right triangle with area $\frac{a^2}{2}$, and therewith legs of length $a$ are
\begin{equation} \label{eq:isosceles_right_ev_dbc}
    \lambda_{m,n} = \frac{\pi^2(m^2+n^2)}{a^2}, \,\, m > n \geq 1.
\end{equation}
The fact that these are \em all \em eigenvalues including multiplicities follows from 
\cite[p. 168]{kuttler1984eigenvalues}; see also \cite{gottlieb1988eigenvalues} and \cite[p. 756]{morse1946methods}.

\subsection{The spectral zeta function and zeta-regularized determinant of isosceles right triangles} \label{ss:zeta_isosceles_right}
Our first result for the isosceles right triangle concerns two equivalent expressions for the spectral zeta function.  To the best of our knowledge these expressions are new.  

\begin{proposition} \label{prop:zeta_det_isosceles_right} The spectral zeta function of an isosceles right triangle with area $a^2/2$ and with the Dirichlet boundary condition is equivalently given by the expressions 
\begin{align*}
    \zetart(s) = &-\frac{1}{4}\left(\frac{a}{\pi}\right)^{2s}\zeta_R(2s) - \frac{1}{2^{s+1}}\left(\frac{a}{\pi}\right)^{2s}\zeta_R(2s) + \frac{\sqrt{\pi}}{4}\left(\frac{a}{\pi}\right)^{2s} \frac{\zeta_R(2s-1)\Gamma(s-1/2)}{\Gamma(s)} \\ 
    &+ \frac{1}{2}\left(\frac{a^2}{\pi}\right)^s\frac{1}{\Gamma(s)}\sum_{n=1}^\infty n^{s-1/2}\sum_{d|n}d^{1-2s} \int_0^\infty x^{s-3/2}e^{-\pi n(x+x^{-1})} dx
\end{align*}
and 
\begin{equation*}
    \zetart(s) = \frac{a^{2s}}{2} \left[\frac{1}{4}G_\diamondsuit(s) - \frac{1}{\pi^{2s}}\zeta_R(2s) - \frac{1}{(2\pi^2)^s}\zeta_R(2s)\right].
\end{equation*}
Here,
\begin{equation}
    G_\diamondsuit(s) = \sum_{m \in \Z} \sum_{n \in \Z} \frac{1}{\pi^{2s}|m + ni|^{2s}}.  \label{eq:Gdiamond}
\end{equation}
\end{proposition}
\begin{proof}
We have
\begin{equation*}
    \zetart(s) = \sum_{\lambda_{m,n}} \frac{1}{\lambda_{m,n}^s} = \left(\frac{a}{\pi}\right)^{2s}\sum_{m>n\geq 1} \frac{1}{(m^2+n^2)^s}.
\end{equation*}
Since
\begin{equation*}
    \sum_{m = 1}^\infty \sum_{n = 1}^\infty \frac{1}{(m^2+n^2)^s} = 2\sum_{m>n\geq 1} \frac{1}{(m^2+n^2)^s} + \frac{\zeta_R(2s)}{2^s},
\end{equation*}
we obtain
\begin{equation} \label{eq:zeta_iso_rect_relation}
    \zetart(s) =
    \frac{1}{2}\zeta_\blacksquare(s) - \frac{1}{2}\left(\frac{a^2}{2\pi^2}\right)^{s}\zeta_R(2s),
\end{equation}
where $\zeta_\blacksquare$ is the spectral zeta function of a square with sides of length $a$. By Proposition \ref{prop:zeta_det_rect},
\begin{align*}
    \zeta_\blacksquare(s) &= \frac{1}{2}\left(\frac{a}{\pi}\right)^{2s}\bigg[-\zeta_R(2s) + \sqrt{\pi}\frac{\zeta_R(2s-1)\Gamma(s-1/2)}{\Gamma(s)}\bigg] \\ &+ \left(\frac{a^2}{\pi}\right)^s\frac{1}{\Gamma(s)}\sum_{n=1}^\infty n^{s-1/2}\sum_{d|n}d^{1-2s} \int_0^\infty x^{s-3/2}e^{-\pi n(x+x^{-1})} dx, \\
    \zeta_\blacksquare(s) &= \frac{a^{2s}}{4}\bigg[G_\diamondsuit(s) - \frac{4}{\pi^{2s}}\zeta_R(2s)\bigg],
\end{align*}
from which the result follows.
\end{proof}

These expressions allow one to meromorphically extend the spectral zeta function to the complex plan and evaluate it at points of interest.  In particular, as a corollary we obtain two expressions for $\zetart'(0)$ which can be used to calculate the zeta-regularized determinant.

\begin{corollary} \label{cor:zeta_det_isosceles}
The zeta-regularized determinant of an isosceles right triangle with area $a^2/2$ and with the Dirichlet boundary condition is $e^{-\zetart'(0)}$ with $\zetart'(0)$ equivalently given by 
\begin{equation*}
    \zetart'(0) = \frac{\log(4a^3)}{4} + \frac{\pi}{24} + \frac{1}{2}\sum_{n=1}^\infty \frac{1}{ne^{2\pi n}}\sum_{d|n} d,
\end{equation*}
\begin{equation*}
    \zetart'(0) = \frac{1}{4}\log\left(\frac{4a^3}{|\eta(i)|^2}\right) = \frac{\log(4a^3)}{4} + \frac{\pi}{24} - \frac{1}{2}\sum_{n=1}^\infty \log(1 - e^{-2\pi n}).
\end{equation*}
\end{corollary}

\begin{proof}
    From \eqref{eq:zeta_iso_rect_relation}, we get
    \begin{equation*}
    \zetart'(s) = \frac{1}{2}\zeta_\blacksquare'(s) - \frac{1}{2}\left(\frac{a^2}{2\pi^2}\right)^{s}\log\left(\frac{a^2}{2\pi^2}\right)\zeta_R(2s) - \left(\frac{a^2}{2\pi^2}\right)^{s}\zeta_R'(2s).
\end{equation*}
By Corollary \ref{cor:zeta_det_rect}, we have the following equivalent expressions
\begin{align*}
    \zetart'(0) &= \frac{1}{2}\left[\frac{1}{2}\log(2a) + \frac{\pi }{12 } + \sum_{n=1}^\infty \frac{1}{ne^{2\pi n}}\sum_{d|n} d\right] - \frac{1}{2}\log\left(\frac{a^2}{2\pi^2}\right)\zeta_R(0) - \zeta_R'(0), \\
    \zetart'(0) &= \frac{1}{4}\log\left(\frac{2a}{|\eta(i)|^2}\right) + \frac{1}{4}\log\left(\frac{a^2}{2\pi^2}\right) + \frac{1}{2}\log(2\pi), 
\end{align*}
which respectively simplify to
\begin{align*}
    \zetart'(0) &= \frac{\log(4a^3)}{4} + \frac{\pi}{24} + \frac{1}{2}\sum_{n=1}^\infty \frac{1}{ne^{2\pi n}}\sum_{d|n} d, \\
    \zetart'(0) &= \frac{1}{4}\log\left(\frac{4a^3}{|\eta(i)|^2}\right).
\end{align*}
\end{proof}

\subsection{The heat trace of isosceles right triangles} \label{s:heat_isoright}
For the isosceles right triangle with area $a^2/2$ and with Dirichlet boundary conditions, the heat trace is 
\begin{align*}
    \heatrtd(t) = \sum_{\lambda_{n,m}} e^{-\lambda_{n,m}t} = \sum_{m>n\geq 1} e^{-\frac{\pi^2(m^2+n^2)}{a^2}t}.
\end{align*}
Let $q = e^{-\frac{\pi^2}{a^2}t}$ and note that
\begin{align*}
    &\sum_{m=1}^\infty \sum_{n=1}^\infty q^{m^2+n^2} = 2\sum_{m>n\geq 1} q^{m^2+n^2} + \sum_{m = 1}^\infty q^{2m^2}.
\end{align*}
Since
\begin{align*}
    \sum_{m=1}^\infty \sum_{n=1}^\infty q^{m^2+n^2} &= \left(\sum_{m=1}^\infty q^{m^2}\right)^2 = \left(\frac{\Theta_3(q)-1}{2}\right)^2, \\
    \sum_{m = 1}^\infty q^{2m^2} &= \frac{\Theta_3(q^2) - 1}{2},
\end{align*}
it follows that
\begin{equation*}
    \heatrtd(t) = \sum_{m>n\geq 1} q^{m^2+n^2} = \frac{1}{2}\left[\left(\frac{\Theta_3(q)-1}{2}\right)^2 -\frac{\Theta_3(q^2) - 1}{2}\right] = \frac{\Theta_3(q)^2 - 2\Theta_3(q) - 2\Theta_3(q^2) + 3}{8}.
\end{equation*}
Here we use our expressions for the heat kernel for the isosceles right triangle to obtain further terms in the short time asymptotic expansion of the heat trace as well as a sharp remainder term. 
\begin{theorem} \label{th:new_heat_isosceles}
    The heat trace with the Dirichlet boundary condition for an isosceles right triangle of area $\frac{a^2}{2}$ admits the asymptotic expansion 
    \begin{align*}
        \heatrtd(t) = \frac{a^2}{8\pi t} - \frac{a(2+\sqrt{2})}{8\sqrt{\pi t}} + \frac{3}{8} - \frac{a}{2\sqrt{2 \pi t}}e^{-\frac{a^2}{2t}}+ \frac{a^2}{2\pi t}e^{-\frac{a^2}{t}}   + \mathcal{O}(t^{-1/2}e^{-\frac{a^2}{t}}), \,\, t \to 0.
    \end{align*}
    The heat trace with the Neumann boundary condition admits the asymptotic expansion
    \begin{align*}
        \heatrtn(t) = \frac{a^2}{8\pi t} + \frac{a(2+\sqrt{2})}{8\sqrt{\pi t}} + \frac{3}{8} + \frac{a}{2\sqrt{2 \pi t}}e^{-\frac{a^2}{2t}} + \frac{a^2}{2\pi t}e^{-\frac{a^2}{t}}   + \mathcal{O}(t^{-1/2}e^{-\frac{a^2}{t}}), \,\, t \to 0.
    \end{align*}
The remainders are sharp.  
\end{theorem}

\begin{proof}
    We have
    \begin{align*}
    \heatrtd(t) = \sum_{m > n \geq 1} e^{-\frac{\pi^2(m^2+n^2)}{a^2}t}.
\end{align*}
Since
\begin{equation*}
    \sum_{m = 1}^\infty\sum_{n=1}^\infty e^{-\frac{\pi^2(m^2+n^2)}{a^2}t} = 2\heatrtd(t) + \sum_{m=1}^\infty e^{-\frac{2\pi^2m^2}{a^2}t}
\end{equation*}
and
\begin{align*}
    \sum_{m=1}^\infty e^{-\frac{\pi^2m^2}{a^2}t} &= \frac{1}{2}\left(\frac{a}{\sqrt{\pi t}} - 1\right) + \frac{a}{\sqrt{\pi t}}\sum_{m=1}^\infty e^{-\frac{m^2a^2}{t}}, \\
    \sum_{m=1}^\infty e^{-\frac{2\pi^2m^2}{a^2}t} &= \frac{1}{2}\left(\frac{a}{\sqrt{2\pi t}} - 1\right) + \frac{a}{\sqrt{2\pi t}}\sum_{m=1}^\infty e^{-\frac{m^2a^2}{2t}}
\end{align*}
by Poisson's summation formula \cite[Lemma 2.2.2]{verhoevencan}, we obtain
\begin{align*}
    \heatrtd(t) &= \frac{1}{2}\left[\sum_{m = 1}^\infty\sum_{n=1}^\infty e^{-\frac{\pi^2(m^2+n^2)}{a^2}t} - \sum_{m=1}^\infty e^{-\frac{2\pi^2m^2}{a^2}t}\right] \\
    &= \frac{1}{2}\left[\left(\frac{1}{2}\left(\frac{a}{\sqrt{\pi t}} - 1\right) + \frac{a}{\sqrt{\pi t}}\sum_{m=1}^\infty e^{-\frac{m^2a^2}{t}}\right)^2 - \frac{1}{2}\left(\frac{a}{\sqrt{2\pi t}} - 1\right) - \frac{a}{\sqrt{2\pi t}}\sum_{m=1}^\infty e^{-\frac{m^2a^2}{2t}}\right] \\
    &= \frac{a^2}{8\pi t} - \frac{a(2+\sqrt{2})}{8\sqrt{\pi t}} + \frac{3}{8} - \frac{a}{2\sqrt{2\pi t}}\sum_{m=1}^\infty e^{-\frac{m^2a^2}{2t}} + \frac{a^2}{2\pi t}\sum_{m=1}^\infty e^{-\frac{m^2a^2}{t}} \\ &- \frac{a}{2\sqrt{\pi t}}\sum_{m=1}^\infty e^{-\frac{m^2a^2}{t}} + \frac{a^2}{2\pi t}\left(\sum_{m=1}^\infty e^{-\frac{m^2a^2}{t}}\right)^2.
\end{align*}
The eigenvalues of an isosceles right triangle with area $a^2/2$ with Neumann boundary conditions are
\begin{equation*}
    \lambda_{m,n} = \frac{\pi^2(m^2+n^2)}{a^2}, \,\, m \geq n \geq 0,
\end{equation*}
so the corresponding heat trace becomes
\begin{align*}
    \heatrtn(t) = \sum_{m \geq n \geq 0} e^{-\frac{\pi^2(m^2+n^2)}{a^2}t} = 1 + \sum_{m=1}^\infty e^{-\frac{\pi^2m^2}{a^2}t} + \sum_{n=1}^\infty e^{-\frac{2\pi^2n^2}{a^2}t} + \heatrtd(t).
\end{align*}
Our expression for $\heatrtd(t)$ then gives 
\begin{align*}
    \heatrtn(t) &= \frac{a^2}{8\pi t} + \frac{a(2+\sqrt{2})}{8\sqrt{\pi t}} + \frac{3}{8} + \frac{a}{2\sqrt{2\pi t}}\sum_{m=1}^\infty e^{-\frac{m^2a^2}{2t}} + \frac{a^2}{2\pi t}\sum_{m=1}^\infty e^{-\frac{m^2a^2}{t}} \\ &+ \frac{a}{2\sqrt{\pi t}}\sum_{m=1}^\infty e^{-\frac{m^2a^2}{t}} + \frac{a^2}{2\pi t}\left(\sum_{m=1}^\infty e^{-\frac{m^2a^2}{t}}\right)^2.
\end{align*}
The proof is now completed by collecting leading order terms.
\end{proof}
Theorem \ref{th:new_heat_isosceles} shows that
\begin{align*}
    \heatrtd(t) &= \frac{a^2}{8\pi t} - \frac{a(2+\sqrt{2})}{8\sqrt{\pi t}} + \frac{3}{8} + \mathcal{O}(e^{-\frac{a^2-\epsilon}{2t}}), \,\, t \to 0, \\
    \heatrtn(t) &= \frac{a^2}{8\pi t} + \frac{a(2+\sqrt{2})}{8\sqrt{\pi t}} + \frac{3}{8} + \mathcal{O}(e^{-\frac{a^2-\epsilon}{2t}}), \,\, t \to 0,
\end{align*}
for any $\epsilon > 0$. Again, $a^2/2$ is the square of half the length of the shortest closed geodesic in $\Omega$ (see \cite[p. 43]{durso1988inverse}).

\section{Spectral invariants of hemi-equilateral (30-60-90) triangles} \label{s:306090}
By \cite{mccartin2003eigenstructure}, the eigenvalues of the 30-60-90 triangle with hypotenuse of length $\ell$ are given by
\begin{equation*}
    \lambda_{m,n} = \frac{4\pi^2}{27r^2}(m^2+mn+n^2) = \frac{16\pi^2}{9\ell^2}(m^2+mn+n^2), \,\, m > n \geq 1.
\end{equation*}
Here, $r$ is the radius of the inscribed circle of the equilateral triangle obtained by doubling the hemi-equilateral triangle.   McCartin shows in \cite{mccartin2003eigenstructure} how antisymmetric eigenfunctions of equilateral triangles form a complete set of eigenfunctions for 30-60-90 triangles \cite{damle2010understanding}; see also \cite[p. 168]{kuttler1984eigenvalues}.

\subsection{The spectral zeta function and zeta-regularized determinant of hemi-equilateral triangles} \label{ss:zeta_306090}
Our first result for these triangles contains two equivalent expressions for the spectral zeta function that to the best of our knowledge are new. 
\begin{proposition} \label{prop:zeta_det_306090} The spectral zeta function of the hemi-equilateral triangle with hypotenuse of length $\ell$ and with the Dirichlet boundary condition is equivalently given by the expressions
\begin{align*}   
    \zetahemi(s)  &= \frac{1}{12} \left(\frac{3\ell}{4\pi}\right)^{2s} \bigg[ -4\zeta_R(2s) - \frac{6}{3^s}\zeta_R(2s) + \frac{2^{2s}\sqrt{\pi}\zeta_R(2s-1)\Gamma(s-1/2)}{\Gamma(s)3^{s-1/2}} \\ 
    &+ \frac{4\pi^s2^{s-1/2}}{\Gamma(s)3^{s/2-1/4}}\sum_{n=1}^\infty n^{s-1/2}\sum_{d|n}d^{1-2s}(-1)^n \int_0^\infty x^{s-3/2}e^{-\pi n\sqrt{3}(x+x^{-1})/2} dx \bigg], \\
    \zetahemi(s) &= \frac{1}{12}\left(\frac{3\ell}{4}\right)^{2s}\left[G_\heartsuit(s) - \frac{6}{\pi^{2s}}\zeta_R(2s) - \frac{6}{(3\pi^2)^s}\zeta_R(2s)\right].
\end{align*}
Here,
\begin{equation}
    G_\heartsuit(s) = \sum_{m \in \Z}\sum_{k \in \Z}\frac{1}{\pi^{2s}|m+kz|^{2s}}, \,\, z = \frac{-3 + i\sqrt{3}}{2}. \label{eq:Gheart}
\end{equation}
\end{proposition}
\begin{proof}
    The spectral zeta function is
\begin{align*}
    \zetahemi(s) = \sum_{\lambda_{m,n}} \frac{1}{\lambda_{m,n}^s} = \left(\frac{3 \ell}{4\pi}\right)^{2s} \sum_{m > n \geq 1} \frac{1}{(m^2+mn+n^2)^s}.
\end{align*}
Since
\begin{align*}
    \sum_{m=1}^\infty\sum_{n=1}^\infty \frac{1}{(m^2+mn+n^2)^s} = 2\sum_{m > n \geq 1} \frac{1}{(m^2+mn+n^2)^s} + \frac{1}{3^s}\zeta_R(2s),
\end{align*}
we can rewrite $\zetahemi$ as
\begin{equation} \label{eq:zeta_eq_306090_relation}
    \begin{split}
        \zetahemi(s) &= \frac{1}{2}\left(\frac{3 \ell}{4\pi}\right)^{2s}\left[\sum_{m=1}^\infty\sum_{n=1}^\infty \frac{1}{(m^2+mn+n^2)^s} - \frac{1}{3^s}\zeta_R(2s)\right] \\
    &= \frac{1}{2}\zetaeq(s) - \frac{1}{2}\left(\frac{3\ell^2}{16\pi^2}\right)^s\zeta_R(2s), 
    \end{split}
\end{equation}
where $\zetaeq$ is the spectral zeta function of the corresponding equilateral triangle with sides of length $\ell$. The result now follows from Proposition \ref{prop:zeta_eqtri}.
\end{proof}

These expressions allow one to meromorphically extend the spectral zeta function to the complex plan and evaluate it at points of interest.  In particular, as a corollary we obtain two expressions for $\zetahemi'(0)$ which can be used to calculate the zeta-regularized determinant.

\begin{corollary} \label{cor:zeta_det_306090}
The zeta-regularized determinant of the hemi-equilateral triangle with hypotenuse of length $\ell$ and with the Dirichlet boundary condition is $e^{-\zetahemi'(0)}$ with $\zetahemi'(0)$ equivalently given by 
\begin{align*}
    \zetahemi'(0) &= \frac{5}{6}\log\left(\frac{\ell}{2}\right) + \frac{7}{12}\log(3) + \frac{\pi\sqrt{3}}{72} + \frac{1}{3}\sum_{n = 1}^\infty \frac{(-1)^n}{ne^{\pi n \sqrt{3}}} \sum_{d | n} d, \\
    \zetahemi'(0) &= \frac 1 3\log\left(\frac{3\ell}{2|\eta(z)|}\right) + \frac{1}{4}\log(3) + \frac{1}{2}\log\left(\frac{\ell}{2}\right), \quad z = \frac{-3 + i\sqrt{3}}{2}.
\end{align*}
\end{corollary}
\begin{proof}
    By \eqref{eq:zeta_eq_306090_relation}, we have
\begin{equation*}
    \zetahemi'(s) = \frac{1}{2}\zetaeq'(s) - \frac{1}{2}\log\left(\frac{3\ell^2}{16\pi^2}\right)\left(\frac{3\ell^2}{16\pi^2}\right)^s\zeta_R(2s) - \left(\frac{3\ell^2}{16\pi^2}\right)^s\zeta_R'(2s).
\end{equation*}
At $s = 0$, this becomes
\begin{equation*}
    \zetahemi'(0) = \frac{1}{2}\zetaeq'(0) + \frac{1}{4}\log(3) + \frac{1}{2}\log\left(\frac{\ell}{2}\right).
\end{equation*}
To complete the proof we apply Corollary \ref{cor:zetadet_eqtri}.
\end{proof}

\subsection{The heat trace of hemi-equilateral triangles} \label{s:heat_hemi}
For the hemi-equilateral triangle with hypotenuse of length $\ell$, we let 
\[ q = e^{-\frac{16\pi^2}{9\ell^2}t}.\]  
Following \cite{verhoevencan} 
\begin{align*}
    \sum_{m=1}^\infty\sum_{n=1}^\infty q^{m^2+mn+n^2} &= \frac{\sum_{m \in \Z}\sum_{n \in \Z}q^{m^2+mn+n^2} - 3\Theta_3(q) + 2}{6} \\
    &= \frac{\Theta_3(q)\Theta_3(q^3) + \Theta_2(q)\Theta_2(q^3) - 3\Theta_3(q) + 2}{6}, 
\end{align*}
and
\begin{align*}
    \sum_{m=1}^\infty\sum_{n=1}^\infty q^{m^2+mn+n^2} &= 2\sum_{m > n \geq 1} q^{m^2+mn+n^2} + \sum_{m = 1} q^{3m} = 2\heathemid(t) + \frac{\Theta_3(q^3)-1}{2},
\end{align*}
we obtain the heat trace for the Dirichlet boundary condition 
\begin{align*}
    \heathemid(t) &= \sum_{m > n \geq 1} e^{-\frac{16\pi^2}{9\ell^2}(m^2+mn+n^2)t} = \frac{1}{2}\left[\sum_{m=1}\sum_{n=1}q^{m^2+mn+n^2} - \frac{\Theta_3(q^3)-1}{2}\right] \\
    &= \frac{\Theta_3(q)\Theta_3(q^3) + \Theta_2(q)\Theta_2(q^3) - 3\Theta_3(q) - 3\Theta_3(q^2) + 5}{12}.
\end{align*}

We calculate the short time asymptotic expansion of the heat trace, obtaining further terms and a sharp remainder.  
\begin{theorem} \label{th:new_heat_306090}
    The heat trace for the hemi-equilateral triangle with hypotenuse of length $\ell$ with the Dirichlet boundary condition admits the asymptotic expansion
    \begin{align*}
        \heathemid(t) = \frac{\ell^2\sqrt{3}}{32\pi t} - \frac{\ell(3+\sqrt{3})}{16\sqrt{\pi t}} + \frac{5}{12} - \frac{\ell}{8}\sqrt{\frac{3}{\pi t}}e^{-\frac{3\ell^2}{16t}} - \frac{3\ell}{8\sqrt{\pi t}}e^{-\frac{9\ell^2}{16t}} + \mathcal{O}(t^{-1}e^{-\frac{3\ell^2}{4t}}).
    \end{align*}
    The heat trace with the Neumann boundary condition admits the asymptotic expansion
    \begin{align*}
        \heathemin(t) = \frac{\ell^2\sqrt{3}}{32\pi t} + \frac{\ell(3+\sqrt{3})}{16\sqrt{\pi t}} + \frac{5}{12} + \frac{\ell}{8}\sqrt{\frac{3}{\pi t}}e^{-\frac{3\ell^2}{16t}} + \frac{3\ell}{8\sqrt{\pi t}}e^{-\frac{9\ell^2}{16t}} + \mathcal{O}(t^{-1}e^{-\frac{3\ell^2}{4t}}).
    \end{align*}
    The remainders are sharp. 
\end{theorem}

\begin{proof}
We have    
\begin{align*}
    \heathemid(t) = \sum_{m > n \geq 1} e^{-\frac{16\pi^2}{9\ell^2}(m^2+mn+n^2)t}.
\end{align*}
Let $\heateqd(t)$ denote the heat trace of an equilateral triangle with side length $\ell$. Since
\begin{equation*}
    \heateqd(t) = \sum_{m = 1}^\infty\sum_{n = 1}^\infty e^{-\frac{16\pi^2}{9\ell^2}(m^2+mn+n^2)t} = 2\heathemid(t) + \sum_{m = 1}^\infty e^{-\frac{16\pi^2}{3\ell^2}m^2t}
\end{equation*}
and
\begin{equation*}
    \sum_{m = 1}^\infty e^{-\frac{16\pi^2}{3\ell^2}m^2t} = \frac{1}{2}\left(\frac{\ell}{4}\sqrt{\frac{3}{\pi t}} - 1\right) + \frac{\ell}{4}\sqrt{\frac{3}{\pi t}}\sum_{m=1}^\infty e^{-\frac{3\ell^2m^2}{16t}}
\end{equation*}
by Poisson's summation formula \cite[Lemma 2.2.2]{verhoevencan}, 
\begin{align*}
    \heathemid(t) = \frac{1}{2}\left[\heateqd(t) - \frac{1}{2}\left(\frac{\ell}{4}\sqrt{\frac{3}{\pi t}} - 1\right) - \frac{\ell}{4}\sqrt{\frac{3}{\pi t}}\sum_{m=1}^\infty e^{-\frac{3\ell^2m^2}{16t}}\right].
\end{align*}
By Theorem \ref{th:eqtri_asy_htrace},
\begin{align*}
    \heathemid(t) &= \frac{1}{2}\bigg[\frac{\ell^2\sqrt{3}}{16\pi t} - \frac{3\ell}{8\sqrt{\pi t}} + \frac{1}{3} + \frac{\ell^2\sqrt{3}}{8\pi t}\sum_{m=1}^\infty e^{-\frac{9\ell^2m^2}{4t}} + \frac{\ell^2\sqrt{3}}{8\pi t}\sum_{n=1}^\infty e^{-\frac{3\ell^2n^2}{4t}} - \frac{3\ell}{4\sqrt{\pi t}}\sum_{m=1}^\infty e^{-\frac{9\ell^2m^2}{16t}}\\
    &+ \frac{\ell^2\sqrt{3}}{4\pi t}\sum_{m=1}^\infty\sum_{n=1}^\infty e^{-\frac{3\ell^2(3m^2+n^2)}{4t}} + \frac{\ell^2\sqrt{3}}{4\pi t}\sum_{m=1}^\infty\sum_{n=1}^\infty e^{-\frac{3\ell^2(3(2m-1)^2+(2n-1)^2)}{16t}} \\ &- \frac{1}{2}\left(\frac{\ell}{4}\sqrt{\frac{3}{\pi t}} - 1\right) - \frac{\ell}{4}\sqrt{\frac{3}{\pi t}}\sum_{m=1}^\infty e^{-\frac{3\ell^2m^2}{16t}}\bigg].
\end{align*}
This simplifies to
\begin{align*}
    \heathemid(t) &= \frac{\ell^2\sqrt{3}}{32\pi t} - \frac{\ell(3+\sqrt{3})}{16\sqrt{\pi t}} + \frac{5}{12} - \frac{\ell}{8}\sqrt{\frac{3}{\pi t}}\sum_{m=1}^\infty e^{-\frac{3\ell^2m^2}{16t}} - \frac{3\ell}{8\sqrt{\pi t}}\sum_{m=1}^\infty e^{-\frac{9\ell^2m^2}{16t}} +\frac{\ell^2\sqrt{3}}{16\pi t}\sum_{n=1}^\infty e^{-\frac{3\ell^2n^2}{4t}}  \\
    &+ \frac{\ell^2\sqrt{3}}{16\pi t}\sum_{m=1}^\infty e^{-\frac{9\ell^2m^2}{4t}} + \frac{\ell^2\sqrt{3}}{8\pi t}\sum_{m=1}^\infty\sum_{n=1}^\infty e^{-\frac{3\ell^2(3m^2+n^2)}{4t}} + \frac{\ell^2\sqrt{3}}{8\pi t}\sum_{m=1}^\infty\sum_{n=1}^\infty e^{-\frac{3\ell^2(3(2m-1)^2+(2n-1)^2)}{16t}}.
\end{align*}
The eigenvalues of the hemi-equilateral triangle with hypotenuse of length $\ell$ and Neumann boundary condition are
\begin{equation*}
    \lambda_{m,n} = \frac{16\pi^2}{9\ell^2}(m^2+mn+n^2), \,\, m \geq n \geq 0,
\end{equation*}
so the heat trace becomes
\begin{align*}
    \heathemin(t) = \sum_{m \geq n \geq 0} e^{-\frac{16\pi^2}{9\ell^2}(m^2+mn+n^2)t} = 1 + \sum_{m=1}^\infty e^{-\frac{16\pi^2}{9\ell^2}m^2t} + \sum_{n=1}^\infty e^{-\frac{16\pi^2}{3\ell^2}n^2t} + \heathemid(t).
\end{align*}
We have
\begin{align*}
    \sum_{m=1}^\infty e^{-\frac{16\pi^2}{9\ell^2}m^2t} &= \frac{1}{2}\left(\frac{3\ell}{4\sqrt{\pi t}} - 1\right) + \frac{3\ell}{4\sqrt{\pi t}} \sum_{m=1}^\infty e^{-\frac{9\ell^2m^2}{16t}}, \\
    \sum_{n=1}^\infty e^{-\frac{16\pi^2}{3\ell^2}n^2t} &= \frac{1}{2}\left(\frac{\ell}{4}\sqrt{\frac{3}{\pi t}} - 1\right) + \frac{\ell}{4}\sqrt{\frac{3}{\pi t}} \sum_{n=1}^\infty e^{-\frac{3\ell^2n^2}{16t}},
\end{align*}
which gives
\begin{align*}
    \heathemin(t) &= \frac{\ell^2\sqrt{3}}{32\pi t} + \frac{\ell(3+\sqrt{3})}{16\sqrt{\pi t}} + \frac{5}{12} + \frac{\ell}{8}\sqrt{\frac{3}{\pi t}}\sum_{m=1}^\infty e^{-\frac{3\ell^2m^2}{16t}} + \frac{3\ell}{8\sqrt{\pi t}}\sum_{m=1}^\infty e^{-\frac{9\ell^2m^2}{16t}} +\frac{\ell^2\sqrt{3}}{16\pi t}\sum_{n=1}^\infty e^{-\frac{3\ell^2n^2}{4t}}  \\
    &+ \frac{\ell^2\sqrt{3}}{16\pi t}\sum_{m=1}^\infty e^{-\frac{9\ell^2m^2}{4t}} + \frac{\ell^2\sqrt{3}}{8\pi t}\sum_{m=1}^\infty\sum_{n=1}^\infty e^{-\frac{3\ell^2(3m^2+n^2)}{4t}} + \frac{\ell^2\sqrt{3}}{8\pi t}\sum_{m=1}^\infty\sum_{n=1}^\infty e^{-\frac{3\ell^2(3(2m-1)^2+(2n-1)^2)}{16t}}.
\end{align*}
\end{proof}
It follows from Theorem \ref{th:new_heat_306090} that
\begin{align*}
    \heathemid(t) = \frac{\ell^2\sqrt{3}}{32\pi t} - \frac{\ell(3+\sqrt{3})}{16\sqrt{\pi t}} + \frac{5}{12} + \mathcal{O}(e^{-\frac{3\ell^2 - \epsilon}{16t}}), \,\, t \to 0, \\
    \heathemin(t) = \frac{\ell^2\sqrt{3}}{32\pi t} + \frac{\ell(3+\sqrt{3})}{16\sqrt{\pi t}} + \frac{5}{12} + \mathcal{O}(e^{-\frac{3\ell^2 - \epsilon}{16t}}), \,\, t \to 0,
\end{align*}
for any $\epsilon > 0$. Once again, $3\ell^2/16$ is the square of half the length of the shortest closed geodesic in the hemi-equilateral triangle with hypotenuse of length $\ell$ as proven in  \cite[p. 43]{durso1988inverse}.

\section{Heat traces of flat tori, convex polygonal domains and a comparison with smoothly bounded domains} \label{s:comparison}
A full-rank lattice $\Gamma \subset \R^n$ is a discrete additive subgroup of the additive group $(\R^n, +)$.  In fact, every discrete additive subgroup of $\R^n$ is a lattice, albeit not necessarily full-rank.  A full-rank lattice $\Gamma \subset \R^n$ gives rise to a smooth, compact Riemannian manifold known as a flat torus, obtained as the quotient $\R^n/\Gamma$. Its Riemannian metric is inherited from the Euclidean (flat) metric on $\R^n$.  The eigenvalues of the Laplacian on the flat torus $\R^n/\Gamma$ are the values $4 \pi^2 ||y||^2$ for all $y$ in the dual lattice $\Gamma^*$, with multiplicities counted according to how many distinct $y$ have the same length; recall 
\[ \Gamma^* = \{ y \in \R^n : y \cdot x \in \Z, \quad \forall x \in \Gamma \}.\] 
The Poisson summation formula is the relation 
\[ \sum_{\gamma \in \Gamma^*} e^{-4\pi t ||\gamma^*||^2} = \frac{\vol(\R^n/\Gamma)}{(4\pi t)^{n/2}} \sum_{\gamma \in \Gamma} e^{-||\gamma||^2/(4t)}. \]
We recognize the left side as the heat trace of the flat torus.  Thus, we have the asymptotic expansion  
\[ \sum_{k \geq 0} e^{-\lambda_k t} = \frac{\vol(\R^n/\Gamma)}{(4\pi t)^{n/2}} \left( 1 + m(\gamma_1) e^{-||\gamma_1||^2/(4t)} + \mathcal O (e^{-||\gamma_2||^2/(4t)}) \right), \quad t \to 0. \]
Above, $\{\lambda_k\}_{k \geq 0}$ are the eigenvalues of the flat torus, $m(\gamma_1)$ is the number of $\gamma \in \Gamma$ of minimal positive length given by $||\gamma_1||$, with the next shortest length given by $||\gamma_2||$.  We then observe that the shortest closed geodesic in $\R^n/\Gamma$ has length $||\gamma_1||$.  Consequently, the asymptotic expansion of the heat trace consists of the usual leading term, together with a remainder term that is of the form $\mathcal O (t^{-n/2} e^{-L^2/(4t)})$ with $L$ the length of the shortest closed geodesic in the flat torus.  This leads us to make a conjecture about the short time asymptotic expansion of the heat trace in similarly flat settings.  

A compact Riemannian manifold with curvature identically equal to zero is known as a Euclidean space form.  The fundamental groups of compact Euclidean space forms are examples of crystallographic groups.  These are discrete groups of Euclidean isometries with compact quotients.  It is interesting to note that in two dimensions, the fundamental domains of crystallographic groups are precisely the integrable polygonal domains of this study.  In two dimensions, all space forms are diffeomorphic to either a flat torus or a Klein bottle. There are 10 diffeomorphism classes of compact 3-dimensional Euclidean space forms, and 75 classes in dimension 4.  Every Euclidean space form is a quotient of a flat torus by a finite group of isometries, and in each dimension there are only finitely many diffeomorphism classes of Euclidean space forms, although the complete classification is known only in low dimensions. We refer to \cite{wolf1972spaces} and \cite{lee2018introduction} for further details about Euclidean space forms.  Due to the vanishing of their curvature, similar to the case of flat tori, we reasonably expect their heat traces to have a similar form. 

\begin{conjecture} \label{conj:spaceform}
Assume that $M$ is an $n$-dimensional Euclidean space form.  Then its heat trace admits an asymptotic expansion of the form 
\[ \sum_{k \geq 0} e^{-\lambda_k t} = t^{-n/2} \left( \frac{\vol(M)}{(4\pi t)^{n/2}} + \mathcal O (e^{-L^2/(4t)}) \right), \quad t \to 0.\]
Here, $L$ is the length of the shortest closed geodesic in $M$. 
\end{conjecture}

In higher dimensions, strictly tessellating polytopes as defined in 
\cite[Definition 1]{rowlett2021crystallographic} are analogous to integrable polygons in dimension two.  Indeed, one could reasonably define an integrable polytope to be a strictly tessellating polytope in the sense of \cite{rowlett2021crystallographic}.  Heuristically, our definition of a polytope is a bounded domain in Euclidean space such that its boundary is piecewise smooth and consists of flat boundary faces.  In two dimensions, for example, a polytope is a bounded, connected polygonal domain.  We suggest that it is reasonable that all polytopes admit a heat trace expansion that behaves analogously to the two-dimensional case.

\begin{conjecture} \label{conj:polytope}
Assume that $M$ is a polytope in $\R^n$.  Then its heat trace with either the Dirichlet or Neumann boundary condition admits an asymptotic expansion of the form 
\[ \sum_{k \geq 0} e^{-\lambda_k t} = t^{-n/2} \left( \sum_{j=0} ^n a_j t^{j/2} + \mathcal O (e^{-c/t}) \right), \quad t \to 0.\]
The coefficient $a_0 = (4\pi)^{-n/2} \vol(M)$ with $\vol(M)$ the $n$-dimensional (Lebesgue) volume of the polytope.  The coefficient $a_1$ can be expressed with a universal constant together with the total $(n-1)$-dimensional volume of the boundary faces of the polytope.  Analogously, the coefficients $a_j$ for $2 \leq j \leq n-1$ can be expressed with a universal constant together with the total $(n-j)$-dimensional volume of the $(n-j)$-dimensional intersections of the boundary faces.  The coefficient $a_n$ can be expressed in terms of the angles in the polytope and its boundary faces as well as angles between the intersections of these.  The infinimum over all $c>0$ such that this remainder estimate holds is $L^2/4$ with $L$ the length of the shortest closed geodesic in $M$.  
\end{conjecture} 

The coefficients $a_j$ for $0 \leq j \leq n-1$ are motivated by locality principles \cite{nursultanov2019heat} that generalize Kac's principle of not feeling the boundary \cite[p. 9]{kac1966can}.  The idea is that on the interior of each $(n-j)$-dimensional subset of the boundary, away from its edges, the heat kernel in $M$ can be modelled after the heat kernel in $\R^{n-j}$.  The leading term in the heat trace then comes simply from the $(n-j)$-dimensional volumes of these subsets, together with certain universal constants.

\subsection{A comparison of heat trace invariants of smoothly bounded domains and polygonal domains} \label{ss:comparison}
We conclude with a comparison of the heat trace expansion of smoothly bounded planar domains to that of polygonal domains that need not be integrable.  We therefore recall the short time asymptotic expansion of the heat trace in these contexts.  
\begin{proposition} \label{prop:heat_trace_expansion_of_smoothly_bounded_domains_and_polygons}
    Let $\Omega \subset \R^2$ be a smoothly bounded domain. For Dirichlet boundary conditions, the heat trace of $\Omega$ satisfies
    \begin{equation*}
        h(t) \sim \frac{a_{-1}}{t} + \frac{a_{-1/2}}{\sqrt{t}} + a_0 + a_{1/2}\sqrt{t}, \,\, t \to 0,
    \end{equation*}
    where
    \begin{equation} \label{eq:heat_trace_coefficients_smoothly_bounded_and_polygons}
        a_{-1} = \frac{|\Omega|}{4\pi}, \quad 
        a_{-1/2} = -\frac{|\partial\Omega|}{8\sqrt{\pi}}, \quad
        a_{0} = \frac{1}{12\pi}\int_{\partial\Omega} k(s) ds, \quad
        a_{1/2} = \frac{1}{256\sqrt{\pi}}\int_{\partial\Omega} k(s)^2 ds, 
\end{equation}
    with $k(s)$ being the Gauss curvature of the boundary. If in addition $\Omega$ is convex, then $a_0 = 1/6$. If instead $\Omega$ is a convex $n$-sided polygon with interior angles $\gamma_1$,\dots,$\gamma_n$, then
    \begin{equation} \label{eq:third_heat_trace_coefficient_polygon}
        a_0 = \sum_{i=1}^n \frac{\pi^2 - \gamma_i^2}{24\pi\gamma_i}.
    \end{equation}
\end{proposition}
\begin{proof}
     The formulas given by \eqref{eq:heat_trace_coefficients_smoothly_bounded_and_polygons} can be found in \cite{watanabe2002plane}. Moreover, by \cite[Thm. 6.10, Remark 6.15]{nursultanov2019heat} we have $a_0 = \frac{\chi(\Omega)}{6}$, which equals $1/6$ if $\Omega$ is convex. Finally, \eqref{eq:third_heat_trace_coefficient_polygon} follows from \cite[Thm. 6.10]{nursultanov2019heat}.
\end{proof}

As a consequence, we will see that the first two heat trace coefficients of a sequence of smoothly bounded convex domains that converge to a convex polygonal domain converge to that of the polygonal domain.  However, the third heat trace coefficient does \em not \em converge to that of the polygonal domain.

\begin{theorem}
    Let $\{\Omega_k\}$ be a sequence of convex smoothly bounded domains in $\R^2$ and let $\Omega$ be a convex polygon such that $\Omega_k \to \Omega$ in the Hausdorff distance. For Dirichlet boundary conditions, heat trace coefficients satisfy 
    \[ a_j (\Omega_k) \to a_j (\Omega), \quad j=-1, -1/2, \quad a_0(\Omega_k) \not\to a_0 (\Omega).\]   
\end{theorem}
\begin{proof}

With the assumptions of convexity and Hausdorff convergence, it follows that the areas $|\Omega_k|$ and perimeters $|\pa \Omega_k|$ converge to the $|\Omega|$ and $|\pa \Omega|$, respectively.  So we now consider the third heat trace coefficient.      By Proposition \ref{prop:heat_trace_expansion_of_smoothly_bounded_domains_and_polygons}, $a_0(\Omega_k) = 1/6$ for every $k$. We will show that $a_0(\Omega) > 1/6$, from which the result follows. If we denote the interior angles of $\Omega$ by $\gamma_1$,\dots,$\gamma_n$, then by Proposition \ref{prop:heat_trace_expansion_of_smoothly_bounded_domains_and_polygons}
\begin{align*}
    a_0(\Omega) &= \sum_{i=1}^n \frac{\pi^2 - \gamma_i^2}{24\pi\gamma_i} = \frac{\pi}{24}\sum_{i=1}^n \frac{1}{\gamma_i} - \frac{1}{24\pi} \sum_{i=1}^n \gamma_i \\ &= \frac{\pi}{24}\sum_{i=1}^n \frac{1}{\gamma_i} - \frac{1}{24\pi} \pi(n-2) = \frac{\pi}{24}\sum_{i=1}^n \frac{1}{\gamma_i} - \frac{n-2}{24}.
\end{align*}
By Cauchy-Schwarz's inequality,
\begin{equation*}
    n^2 \leq \sum_{i=1}^n \frac{1}{\gamma_i} \sum_{i=1}^n \gamma_i = \sum_{i=1}^n \frac{1}{\gamma_i} \pi(n-2),
\end{equation*}
so that
\begin{equation*}
    \sum_{i=1}^n \frac{1}{\gamma_i} \geq \frac{n^2}{\pi(n-2)}.
\end{equation*}
Thus,
\begin{equation}
    a_0(\Omega) \geq \frac{\pi}{24} \frac{n^2}{\pi(n-2)} - \frac{n-2}{24} 
    = \frac{1}{6} + \frac{1}{6(n-2)} > \frac{1}{6}. \label{eq:a0_lower}
\end{equation}
\end{proof}

It is interesting to note that if instead we approximate a smoothly bounded domain by polygonal domains, this third heat trace coefficient of the polygonal domains \em converges \em to that of the smoothly bounded domain.  

\begin{theorem}[See also \cite{maardby2023mathematics}]
    Let $\{\Omega_k\}$ be a sequence of $N_k$-sided convex polygons with interior angles $\gamma_{k,j}$, for $ k \geq 1$ and $1 \leq j \leq N_k$. Assume that $\overline{\Omega_k} \to \overline{\Omega}$ in Hausdorff, with $\Omega$ being a nonempty smoothly bounded convex domain. Then the first three heat trace coefficients of $\Omega_k$ converge to those of $\Omega$. 
\end{theorem}
\begin{proof}
The first two heat trace coefficients converge thanks to the assumptions of Hausdorff convergence and convexity.  By \cite[Lemma 4.7]{gmjr_survey}, the interior angles $\gamma_{k,j}$ all tend to $\pi$ as the polygons tend to the smoothly bounded domain in Hausdorff convergence. Next, we show that $N_k \to \infty$ as $k \to \infty$. Suppose instead that there is an $M > 0$ such that $N_k \leq M$ for all $k$. Since the angles all tend to $\pi$, there is an $N \geq 1$ such that $\gamma_{k,j} > \pi - \frac{2\pi}{M}$ for all $k \geq N$ and $1 \leq j \leq N_k$. Then, for $k \geq N$,
\begin{equation*}
    \pi(N_k - 2) = \sum_{j = 1}^{N_k} \gamma_{k,j} > N_k\left(\pi - \frac{2\pi}{M}\right),
\end{equation*}
which implies that $N_k > M$, a contradiction. 

Now, the term $a_0$ for each $k$ is
\[ a_0 (\Omega_k) = \frac \pi {24} \sum_{k=1} ^{N_k} \frac{1}{\gamma_{k,j}} - \frac{N_k}{24} + \frac {1}{12}.\]
Following the proof of \cite[Thm. 4.4.1]{maardby2023mathematics}, we can write $\gamma_{k,j} = \pi(1-f(k,j))$, $k \geq 1, \,\, 1 \leq j \leq N_k$, from which it follows that $\sum_{j=1}^{N_k} f(k,j) = 2$ for every $k$ and
\begin{equation*}
    a_0(\Omega_k) = \frac{1}{6} + \frac{1}{24}\sum_{j=1}^{N_k} \frac{f(k,j)^2}{1-f(k,j)}.
\end{equation*}
If we then write 
\begin{equation*}
    \epsilon_k = \max_{1 \leq j \leq N_k} f(k,j),
\end{equation*}
then $\epsilon_k \to 0$ because the angles tend to $\pi$.  We therefore obtain that 
\begin{equation*}
  0 \leq  \sum_{j=1}^{N_k} \frac{f(k,j)^2}{1-f(k,j)} \leq \frac{\epsilon_k}{1-\epsilon_k} \sum_{j=1}^{N_k} f(k,j) = \frac{2\epsilon_k}{1-\epsilon_k} \to 0 \textrm{ as } k \to \infty.
\end{equation*}
Thus, $a_0 (\Omega_k) \to \frac 1 6$ $= a_0 (\Omega)$ as $k \to \infty$.  
\end{proof}

There are numerous modes of geometric convergence for domains, Riemannian manifolds, and more general types of possibly singular spaces.  Under different modes of convergence, one can study the behavior of spectral invariants of a sequence and compare to those of the limit space, as long as it is possible to define a Laplace spectrum on the elements of the sequence and also on the limit space.  Interestingly, one can define a Laplace spectrum on very singular spaces, including but not limited to noncollapsed limits under Gromov-Hausdorff convergence \cite{cc3}, rough Riemannian manifolds \cite{roughweyl}, and RCD spaces \cite{rcdweyl}.  In some cases, it is even possible to define notions of curvature, from which one could hope to obtain higher order heat trace invariants.  As a first step, one could investigate the convergence of the most elementary spectral invariants:  the individual eigenvalues.  Convergence of individual eigenvalues under Gromov-Hausdorff convergence to noncollapsed limits of compact manifolds with Ricci curvature bounded below was shown by Cheeger and Colding \cite{cc3}.  In the same setting, the associated heat kernels also converge \cite{ding}.  However, the convergence of other spectral invariants can be much more subtle, because in essence it could involve several limiting processes that need not commute.  In the simple setting of Hausdorff convergence of planar domains, if we remove the assumption of convexity, a quantity as simple as the perimeters of the domains need not converge!  There are many interesting problems one could study in the general field of spectral geometry, exploring  relationships between the Laplace spectrum and the underlying geometry, and we welcome both newcomers and seasoned researchers to join us in exploring!

\begin{appendix} \section{Estimates} \label{appendix}
Here we show that certain quantities are bounded and therewith justify our calculations of the zeta-regularized determinants.

\begin{lemma} \label{lemma_vanishes}
For any $a,b > 0$, the quantity 
\[    \frac{d}{ds}\left[\left(\frac{ab}{\pi}\right)^s\sqrt{\frac{a}{b}}\sum_{n=1}^\infty n^{s-1/2}\sum_{d|n}d^{1-2s} \int_0^\infty x^{s-3/2}e^{-\pi an(x+x^{-1})/b} dx\right].\]
is bounded in a neighborhood of $s=0$. In particular,
\begin{equation*}
    \lim_{s \to 0} \frac{1}{\Gamma(s)}\frac{d}{ds}\left[\left(\frac{ab}{\pi}\right)^s\sqrt{\frac{a}{b}}\sum_{n=1}^\infty n^{s-1/2}\sum_{d|n}d^{1-2s} \int_0^\infty x^{s-3/2}e^{-\pi an(x+x^{-1})/b} dx\right] = 0.
\end{equation*}
\end{lemma} 
\begin{proof}
Since
\begin{align*}
    &\frac{d}{ds}\left[\left(\frac{ab}{\pi}\right)^s\sqrt{\frac{a}{b}}\sum_{n=1}^\infty n^{s-1/2}\sum_{d|n}d^{1-2s} \int_0^\infty x^{s-3/2}e^{-\pi an(x+x^{-1})/b} dx\right] \\
    &= \left(\frac{ab}{\pi}\right)^s\log\left(\frac{ab}{\pi}\right)\sqrt{\frac{a}{b}}\sum_{n=1}^\infty n^{s-1/2}\sum_{d|n}d^{1-2s} \int_0^\infty x^{s-3/2}e^{-\pi an(x+x^{-1})/b} dx \\
    &+ \left(\frac{ab}{\pi}\right)^s\sqrt{\frac{a}{b}}\frac{d}{ds}\left[\sum_{n=1}^\infty n^{s-1/2}\sum_{d|n}d^{1-2s} \int_0^\infty x^{s-3/2}e^{-\pi an(x+x^{-1})/b} dx\right],
\end{align*}
it is enough to show that
\begin{equation*}
    \frac{d}{ds}\left[\sum_{n=1}^\infty n^{s-1/2}\sum_{d|n}d^{1-2s} \int_0^\infty x^{s-3/2}e^{-\pi na(x+x^{-1})/b} dx\right]
\end{equation*}
is bounded in a neighborhood of zero, say $s \in (-1, 1)$. We will in fact show that we may differentiate termwise and differentiate under the integral sign, from which the lemma will follow. Let
\begin{equation*}
    f_N(s) = \sum_{n=1}^N n^{s-1/2}\sum_{d|n}d^{1-2s} \int_0^\infty x^{s-3/2}e^{-\pi an(x+x^{-1})/b} dx, \,\, N \geq 1, \,\, s \in (-1,1).
\end{equation*}
By definition of infinite sums, $f_N$ converges pointwise to
\begin{equation*}
    f(s) = \sum_{n=1}^\infty n^{s-1/2}\sum_{d|n}d^{1-2s}\int_0^\infty x^{s-3/2}e^{-\pi an(x+x^{-1})/b} dx.
\end{equation*}
Now, we need to show that $f_N'(s)$ converges uniformly to some function. We have
\begin{align*}
    f_N'(s) &= \sum_{n=1}^N n^{s-1/2}\sum_{d|n}d^{1-2s}(\log(n) - 2\log(d))\int_0^\infty x^{s-3/2}e^{-\pi an(x+x^{-1})/b} dx \\
    &+ \sum_{n=1}^N n^{s-1/2}\sum_{d|n}d^{1-2s} \frac{d}{ds}\left[\int_0^\infty x^{s-3/2}e^{-\pi an(x+x^{-1})/b} dx\right].
\end{align*}
To proceed, we want to show that we may differentiate under the integral sign. Let 
\begin{equation*}
    h(s,x) = x^{s-3/2}e^{-\pi an(x+x^{-1})/b}, \,\, s \in (-1,1), \,\, x \in (0,\infty).
\end{equation*}
Fix $s$. Since $h(s,x) \to 0$ as $x \to 0$ and $h(s,x)$ decays exponentially as $x \to \infty$, it follows that $h(s,x)$ is Lebesgue-integrable over $x \in (0,\infty)$. Moreover,
\begin{equation*}
    \frac{\partial h}{\partial s} = \log(x)x^{s-3/2}e^{-\pi an(x+x^{-1})/b}
\end{equation*}
exists for all $s \in (-1,1)$ and $x \in (0,\infty)$. Finally, let
\begin{equation*}
    \theta(x) = 
    \begin{cases}
        \log(x)x^{-5/2}e^{-\pi an(x+x^{-1})/b}, \,\, 0 < x < 1, \\
        \log(x)e^{-\pi an(x+x^{-1})/b}, \,\, x \geq 1.
    \end{cases}
\end{equation*}
By construction we have $|\frac{\partial h}{\partial s}| \leq \theta(x)$ for all $s \in (-1,1)$ and $x \in (0,\infty)$, and $\theta(x)$ is Lebesgue-integrable over $x \in (0,\infty)$ by the same arguments as for $h(s,x)$. Thus, it follows from \cite[Thm. 2.27]{folland1999real} that
\begin{equation*}
    \frac{d}{ds}\left[\int_0^\infty x^{s-3/2}e^{-\pi an(x+x^{-1})/b} dx\right] = \int_0^\infty \log(x)x^{s-3/2}e^{-\pi an(x+x^{-1})/b} dx.
\end{equation*}
Therefore,
\begin{align*}
    f_N'(s) &= \sum_{n=1}^N n^{s-1/2}\sum_{d|n}d^{1-2s}(\log(n) - 2\log(d)) \int_0^\infty x^{s-3/2}e^{-\pi an(x+x^{-1})/b} dx \\
    &+ \sum_{n=1}^N n^{s-1/2}\sum_{d|n}d^{1-2s} \int_0^\infty \log(x)x^{s-3/2}e^{-\pi an(x+x^{-1})/b} dx.
\end{align*}
To show that $f_N'(s)$ converges uniformly, we use Weierstrass' M-test. Write $f_N'(s) = \sum_{n=1}^N g_n(s)$ where
\begin{align*}
    g_n(s) = \sum_{d|n}n^{s-1/2}d^{1-2s}\bigg[(\log(n) - 2\log(d)) \int_0^\infty x^{s-3/2}e^{-\pi an(x+x^{-1})/b} dx \\
    + \int_0^\infty \log(x)x^{s-3/2}e^{-\pi an(x+x^{-1})/b} dx\bigg].
\end{align*}
We need to bound $|g_n(s)|$ by some sequence $M_n$ such that $\sum_{n=1}^\infty M_n$ converges. We have
\begin{align*}
    |g_n(s)| &\leq \sum_{d|n}n^{s-1/2}d^{1-2s}\bigg|(\log(n) - 2\log(d)) \int_0^\infty x^{s-3/2}e^{-\pi an(x+x^{-1})/b} dx \\
    &+ \int_0^\infty \log(x)x^{s-3/2}e^{-\pi an(x+x^{-1})/b} dx\bigg| \\
    &\leq n^{9/2}\bigg[3\log(n) \int_0^\infty x^{s-3/2}e^{-\pi an(x+x^{-1})/b} dx + \int_0^\infty |\log(x)|x^{s-3/2}e^{-\pi an(x+x^{-1})/b} dx\bigg].
\end{align*}
To obtain a bound on the first integral, we compute
\begin{align*}
    \int_0^\infty x^{s-3/2}e^{-\pi an(x+x^{-1})/b} dx &= \int_0^1 x^{s-3/2}e^{-\pi an(x+x^{-1})/b} dx + \int_1^\infty x^{s-3/2}e^{-\pi an(x+x^{-1})/b} dx \\
    &\leq \int_0^1 x^{-5/2}e^{-\pi an(x+x^{-1})/b} dx + \int_1^\infty e^{-\pi an(x+x^{-1})/b} dx \\
    &\leq \int_0^1 x^{-3}e^{-\pi anx^{-1}/b} dx + \int_1^\infty e^{-\pi anx/b} dx \\
    &= \frac{b}{\pi a n}\left(2 + \frac{b}{\pi a n}\right)e^{-\pi a n/b}.
\end{align*}
Similarly for the second integral,
\begin{align*}
    \int_0^\infty |\log(x)|x^{s-3/2}e^{-\pi an(x+x^{-1})/b} dx &\leq \int_0^1 |\log(x)|x^{-3}e^{-\pi anx^{-1}/b} dx + \int_1^\infty \log(x)e^{-\pi anx/b} dx \\
    & \leq \int_0^1 x^{-4}e^{-\pi anx^{-1}/b} dx + \int_1^\infty xe^{-\pi anx/b} dx \\ &= \frac{b}{\pi an}\left(2 + \frac{3b}{\pi an} + \frac{2b^2}{(\pi an)^2}\right)e^{-\pi an/b}.
\end{align*}
Thus,
\begin{equation} \label{eq:g_n_estimate}
    |g_n(s)| \leq n^{9/2}\bigg[\frac{3b}{\pi a }\left(2 + \frac{b}{\pi a n}\right) + \frac{b}{\pi an}\left(2 + \frac{3b}{\pi an} + \frac{2b^2}{(\pi an)^2}\right)\bigg]e^{-\pi a n/b}.
\end{equation}
In particular, there are constants $C > 0$ and $M \geq 1$ such that $|g_n(s)| \leq Cn^Me^{-\pi an/b}$ for all $n \geq 1$ and $s \in (-1,1)$. Since
\begin{equation*}
    \sum_{n = 1}^\infty Cn^Me^{-\pi an/b}
\end{equation*}
converges, it follows from Weierstrass' M-test that $f_N'(s)$ converges uniformly on $(-1,1)$. This in turn implies that we can differentiate $f$ termwise (see e.g. \cite[Thm. 7.17]{rudin1964principles}), i.e.
\begin{align*}
    &\frac{d}{ds}\left[\sum_{n=1}^\infty n^{s-1/2}\sum_{d|n}d^{1-2s} \int_0^\infty x^{s-3/2}e^{-\pi an(x+x^{-1})/b} dx\right] \\
    &= \sum_{n=1}^\infty n^{s-1/2}\sum_{d|n}d^{1-2s}(\log(n)-2\log(d)) \int_0^\infty x^{s-3/2}e^{-\pi an(x+x^{-1})/b} dx \\
    &+ \sum_{n=1}^\infty n^{s-1/2}\sum_{d|n}d^{1-2s} \int_0^\infty \log(x)x^{s-3/2}e^{-\pi an(x+x^{-1})/b} dx.
\end{align*}
In particular, we can by \eqref{eq:g_n_estimate} conclude that the derivative is bounded for $s \in (-1,1)$.

\end{proof} 
\end{appendix}

\end{document}